\documentclass[12 pt]{article}
\usepackage{amsmath}
\usepackage{amsthm}
\usepackage{graphicx}
\usepackage{amssymb}
\usepackage{xcolor}

\def\dl{\delta}

\def\ep{\epsilon}

\def\Im{{\rm{Im}}\,}

\def\G{\Gamma}

\DeclareMathOperator{\sgn}{sgn}

\DeclareMathOperator{\ran}{ran}
\DeclareMathOperator{\Tr}{Tr}

\newcommand{\R}{{\mathbb{R}}}
\newcommand{\C}{{\mathbb{C}}}
\newcommand{\Z}{{\mathbb{Z}}}
\newcommand{\N}{{\mathbb{N}}}

\newcommand{\T}{{\mathbb{T}}}

\newtheorem{df}{Definition}

\newtheorem{lem}{Lemma}
\newtheorem{prop}{Proposition}
\newtheorem{rem}{Remark}
\newtheorem{cor}{Corollary}

\newtheorem{thm}{Theorem}

\title{Stability of traveling waves for the Burgers--Hilbert equation}
\author{\'Angel Castro\footnote{Corresponding author: angel\_castro@icmat.es}, Diego C\'ordoba\footnote{dcg@icmat.es} and Fan Zheng\footnote{fan.zheng@icmat.es} \\ \\
\small Instituto de Ciencias Matem\'aticas ICMAT-CSIC-UAM-UCM-UC3M\\ \small 28049, Madrid, Spain.
}

\date{}

\begin{document}
\maketitle

\begin{abstract}
We consider smooth solutions of the Burgers-Hilbert equation that are
a small perturbation  $\delta$ from a global periodic traveling wave with small amplitude  $\epsilon$. We use a modified energy method to prove the existence time of smooth solutions on a time scale of $\frac{1}{\epsilon\delta}$ with $0<\delta\ll\ep\ll1$ and on a time scale of $\frac{\epsilon}{\delta^2}$ with $0<\delta\ll\ep^2\ll1$. Moreover, we show that the traveling wave exists for an amplitude  $\ep$ in the range $(0,\ep^*)$ with $\ep^*\sim 0.23$ and fails to exist for $\ep>\frac{2}{e}$.
\end{abstract}

\vskip 0.3cm


\section{Introduction}
\subsection{The Burger-Hilbert equation (BH).} In this paper we study the size and stability of traveling waves of the Burgers-Hilbert equation (BH),
\begin{align}
&f_t =  Hf + ff_x, \qquad \text{for $(x,t)\in \Omega\times \R$}\label{BH}\\
&f(x,0)= f_0(x).
\end{align}
where $\Omega$ is the real line $\R$ or the torus $\T=\R/2\pi\Z$ and $Hf$ is the Hilbert transform which is defined for $f : \R$ (resp. $\T) \longrightarrow \R$ by
\[
Hf(x) = \frac{1}{\pi} P.V. \int_{\R}\frac{f(y)}{x - y} dy \quad\text{resp. }
Hf(x) = \frac{1}{2\pi} P.V. \int_0^{2\pi} f(y)\cot\frac{x - y}{2} dy.
\]
Its action in the frequency space is $\widehat{Hf}(k) = -i\sgn k\hat f(k)$,
for $k \neq 0$, and $\widehat{Hf}(0) = 0$.

This equation arised in \cite{Marsden-W} as a quadratic approximation for the evolution of the boundary of a simply connected vorticity patch in 2D. Later, Biello and Hunter, in \cite{Biello}, proposed the model as an approximation for describing the dynamics of small slope vorticity fronts in the two-dimensional incompressible Euler equations. Recently, the validity of this approximation is proved in \cite{Hunter-MSZ}.

By standard energy estimates the initial value problem for (BH) is locally well posed in $H^s$ for $s>\frac32$. Bressan and Nguyen established in \cite{Bressan-N} global existence of weak solutions for initial data $f_0\in L^2(\R)$ with $f(x,t)\in L^\infty(\R)\cap L^2(\R)$ for all $t>0$.  Bressan and Zhang constructed, in \cite{Bressan-Z}, locally in time piecewise continuous solutions to the BH equation with a single discontinuity where the Hilbert transform generates a logarithmic singularity. Uniqueness for general global weak solutions of \cite{Bressan-N} is open. But piecewise continuos solutions are shown to be unique in \cite{vasseur} by Krupa and Vasseur.

Burgers--Hilbert equation can indeed form shocks in finite time. Various numerical simulations have been performed in \cite{Biello,Hunter,Klein}.
Finite time singularities, in the $C^{1,\delta}$ norm with $0<\delta<1$,  was shown to exist in \cite{CCG} for initial data $f_0$ in $L^2(\R)\cap C^{1,\delta}(\R)$
that has a point $x_0\in \R$ such that $H(f_0)(x_0) > 0$ and $f_0(x_0) \geq (32\pi||f_0||_{L^2})^{\frac13}$. Recently, with a different approach, Saut and Wang \cite{Saut-W} proved shock formation in finite time for (BH) and Yang \cite{Yang} constructed solutions that develop an asymptotic self-similar shock at one single point with an explicitly computable blowup profile for (BH).

In this paper we are concerned with the dynamics in the small amplitude regime where (BH) can be viewed as a perturbation of the linearized (BH) equation $f_t = H[f]$. Since the nonlinear term in \eqref{BH} is quadratic and the Hilbert transform is orthogonal in $L^2$, standard energy estimates yield a time of existence of smooth solutions $T\sim \frac{1}{||f_0||}$. Thanks to the effect of the Hilbert transform and using the normal form method, Hunter, Ifrim, Tataru and Wong (see \cite{Hunter-I} and \cite{Hunter-I-T}) were able to improve this time of existence. More precisely, if $\ep$ is the size of the initial data, they prove a lifespan $T\sim \frac{1}{\ep^2}$ for small enough $\ep$ (see also \cite{EW} for a similar approach with a modified version of  the (BH) equation). The proofs are based on the normal form method and in the modified energy method. Furthermore, Hunter \cite{Hunter} showed for $0<\ep \ll 1$ the existence  of $C^\infty$-traveling wave solutions of the form  $$f_\ep(x,t)=u_\ep(x+v_\ep t)$$ with
\begin{align}
u_\ep(x)&=\ep \cos(x)+O(\ep^2),\label{travelling}\\
v_\ep & = -1+O(\ep^2)\label{speed},
\end{align}
Notice that, $\left(\frac{1}{n}u_\ep(nx),\frac{1}{n}v_\ep\right)$ is also a $C^\infty-$traveling wave solution.

Throughout the paper we will assume that the initial data $f_0$ has zero mean. Since the equation (\ref{BH}) preserves the mean,
\[
\int_0^{2\pi}f(x,t)dx=0
\]
for all $t$. Since in the construction above $u_\ep$ also has zero mean,
\[
\int_0^{2\pi}f(x,t)dx=0
\]
for all $t$.

\subsection{The main theorem.} In the present work we extend the results in the small amplitude regime in the following way:
\begin{enumerate}
\item \underline{Size of the traveling waves:} we show that the traveling wave exist for an amplitude  $\ep$ in the range $(0,\ep^*)$ with $\ep^*\sim 0.23$ and fails to exist for $\ep>\frac{2}{e}$.
\item \underline{Extended lifespan from a traveling wave:} we prove that a $\delta-$perturbation of $u_\ep$ lives, at least, for a time $T\sim \frac{1}{\delta\ep}$, for $0<\delta\ll\ep\ll1$, and for a time $T\sim \frac{\epsilon}{\delta^2}$ for $0<\delta\ll\ep^2\ll1$. This is an improvement compared with the time  $T\sim \frac{1}{\ep^2}$ provided by the results in \cite{Hunter-I} and \cite{Hunter-I-T}. Indeed, our main theorem reads:

\end{enumerate}
\begin{rem}
After publication of this paper, Dahne and G\'omez-Serrano have found out exactly the maximum amplitude of the traveling waves, which corresponds to $\ep = 0.54771699\dots$ in our parametrization \cite{DaGS}.
\end{rem}

\begin{thm}\label{mainthm}For $0<|\ep|,\, \delta\ll1$ let $(u_\ep, v_\ep)\in  C^\infty(\T)\times \R$  be a traveling wave solution of equation \eqref{BH} as in \eqref{travelling} and \eqref{speed} and $$||f_0-u_\ep||_{H^4(\T)}<\delta.$$ Then  there exist, $0<\ep_0\ll1$, $T(\ep,\delta)>0$ and a solution of equation \eqref{BH}
\begin{align*}
f(x,t)\in C\left([0,T(\ep,\delta));\, H^{4}(\T)\right),
\end{align*}
such that:
\begin{enumerate}
\item  if $\delta\ll|\ep|$ and $|\ep|\leq \ep_0$,  $T(\ep,\delta)\sim \frac{1}{\ep\delta},$
\item if $\delta\ll\ep^2$ and $|\ep|\leq \ep_0$,  $T(\ep,\delta)\sim \frac{\ep}{\delta^2}.$
\end{enumerate}
Moreover, there are two differentiable functions $\ep(t)$ and $a(t)$ such that
\[
\|f(x, t) - u_{\ep(t)}(x + a(t))\|_{H^4} \lesssim \delta.
\]
\end{thm}


\subsection{Sketch of the proof of theorem \ref{mainthm}} Now we briefly describe the proof of theorem \ref{mainthm}.
Assume that the solution
\[
f(x, t) = u_\ep(x + v_\ep t) + g(x + v_\ep t, t)
\]
is a small perturbation around the traveling wave $u_\ep(x + v_\ep t)$.
Then the linearization of the Burgers-Hilbert equation (\ref{BH}) is
\[
L_\ep g := -v_\ep g_x + Hg + (u_\ep(x)g)_x = 0
\]
so to the first order, the perturbation $g$ solves the equation $g_t = L_\ep g$, with solution
\[
g(x, t) = e^{tL_\ep}g(x, 0).
\]
Therefore the linear evolution of $g$ is determined by the eigenvalues of $L_\ep$.

The full nonlinear evolution of $g$ is
\[
g_t = L_\ep g + N(g, g)
\]
where $N(g, g)$ is a nonlinearity that is (at least) quadratic in $g$.
We plug in the linear solution to get
\[
g_t = e^{tL_\ep}L_\ep g(x, 0) + N(e^{tL_\ep}g(x, 0), e^{tL_\ep}g(x, 0))
\]
to second order, which integrates to
\[
g(x, t) = e^{tL_\ep}g(x, 0) + e^{tL_\ep}\int_0^t e^{-sL_\ep}N(e^{sL_\ep}g(x, 0), e^{sL_\ep}g(x, 0))ds.
\]
Expand (at least formally) the initial data and the nonlinearity in terms of the eigenvectors of $L_\ep$ as
\[
g(x, 0) = \sum_n c_n\varphi_n(x), \quad N(\varphi_l, \varphi_m) = \sum_n c_{lmn} \varphi_n,
\]
where the eigenvalue of $\varphi_n$ is $\lambda_n$. Then
\begin{equation}\label{2nd-order}
g(x, t) \approx \sum_n c_ne^{\lambda_nt}\varphi_n(x) + \sum_{l,m,n} \frac{e^{(\lambda_l+\lambda_m)t} - e^{\lambda_nt}}{\lambda_l + \lambda_m - \lambda_n}c_{lmn}\varphi_n(x)
\end{equation}
to second order, provided that the denominator $\lambda_l + \lambda_m - \lambda_n \neq 0$, i.e., that the eigenvalues are ``non-resonant".
Then we can integrate the equation (\ref{BH}) up to a cubic error term,
yielding the ``cubic lifespan", i.e., initial data of size $\ep$ leads to a solution that exists for a time at least comparable to $\ep^{-2}$.
This is the ``normal form transformation", first proposed by Poincar\'e in the setting of ordinary differential equations (see \cite{Arnold} for a book reference). Its application to partial differential equations was initiated by Shatah \cite{Shatah} in the study of the nonlinear Klein--Gordon equation, and then extended to the water wave problem by Germain--Masmoudi--Shatah \cite{GeMaSh2,GeMaSh3} and Ionescu--Pusateri \cite{IoPu,IoPu2}, the Burgers--Hilbert equation by Hunter--Ifrim--Tataru--Wang \cite{Hunter-I-T}, and more recently, the Einstein--Klein--Gordon equation \cite{IoPaEKG} by Ionescu--Pausader.

Unfortunately, non-resonance fails for $L_\ep$ because 0 is an eigenvalue,
and $0 + \lambda_n - \lambda_n = 0$.
The eigenvalue 0 arises from the symmetry of the equation (\ref{BH}).
Indeed, the initial data $u_\ep(x + \delta) \approx u_\ep(x) + \delta u_\ep'(x)$ produces the solution
\[
f(x, t) = u_\ep(x + v_\ep t + \delta) \approx u_\ep(x + v_\ep t) + \delta u_\ep'(x + v_\ep t).
\]
In this case $g(x, t) = \delta u_\ep'(x)$ with $g_t = 0$, so $u_\ep' \in \ker L_\ep$. Also, the initial data $u_{\ep+\delta}(x) \approx u_\ep(x) + \delta\partial_\ep u_\ep(x)$ produces the solution
\[
f(x, t) = u_{\ep+\delta}(x + v_{\ep+\delta}t) \approx u_\ep(x + v_\ep t) + \delta\partial_\ep u_\ep(x + v_\ep t) + \delta v_\ep'tu_\ep'(x + v_\ep t).
\]
In this case $g(x, t) = \delta\partial_\ep u_\ep(x) + \delta v_\ep'tu_\ep'(x)$, so
\[
L_\ep g = \delta L_\ep\partial_\ep u_\ep = g_t = \delta v_\ep'u_\ep' \in \ker L_\ep
\]
so $\partial_\ep u_\ep$ is in the generalized eigenspace corresponding to the eigenvalue 0.

These perturbations generate translations and variations along the bifurcation curve. We treat them separately using a more sophisticated ansatz
\[
f(x, t) = u_{\ep(t)}(x + a(t)) + g(x + a(t), t).
\]
We will show in Proposition \ref{new-frame} that if $|\ep_0|$ and $\|f - u_{\ep_0}\|_{H^2}/|\ep_0|$ are sufficiently small, then $f$ can always be put in the form above, with $|\ep - \ep_0|/|\ep_0|$ also small, and the expansion of $g$ not involving any eigenvector with eigenvalue 0. This way we removed the resonance caused by the eigenvalue 0 from the evolution of $g$.

We also need to analyze the other eigenvalues of $L_\ep$,
a first order differential operator with variable coefficients,
and a quasilinear perturbation from $L_0 = \partial_x + H$,
whose eigenvectors are the Fourier modes $e^{inx}$.
Just like the Schr\"odinger operator with potential $-\Delta + V$,
with a basis of eigenvectors known as the ``Jost functions",
giving rise to the ``distorted Fourier transform" (see Agmon \cite{Agmon}), $L_\ep$ can also be diagonalized using a combination of conjugation and perturbative analysis. More precisely, let $g = h_x$. Then
\[
L_\ep g = ((u_\ep(x) - v_\ep)g)_x + Hg = ((u_\ep(x) - v_\ep)h_x + Hh)_x
\]
so $L_\ep$ is conjugate to the operator $h \mapsto (u_\ep(x) - v_\ep)h_x + Hh$. Let $h = \tilde h \circ \phi_\ep$ where $\phi_\ep'(x)$ is proportional to $(u_\ep(x) - v_\ep)^{-1}$. Then
\[
L_\ep g = ((c_\ep\partial_x + H + R_\ep)\tilde h \circ \phi_\ep)_x
\]
where $c_\ep \to 1$ as $\ep \to 0$, and $R_\ep$ is a small smoothing remainder (i.e., it gains derivatives of arbitrarily high orders).
Thus $L_\ep$ is conjugate to $c_\ep\partial_x + H + R_\ep$,
whose eigenvalues can be approximated by those of $c_\ep\partial_x + H$,
which are $\pm(nc_\ep i - i)$, $n = 1, 2, \dots$. The general theory of unbounded analytic operators developed in \cite{kato} allows us to justify this approximation up to $O(\ep^6)$ (see Corollary \ref{lambda-remainder}), and to relate the eigenvectors of $L_\ep$ to the Fourier modes (see Lemma \ref{conjugado}), in the sense that another linear map $\tilde h \mapsto \mathfrak h$ conjugates $L_\ep$ into a Fourier multiplier whose action on $e^{i(n+\sgn n)x}$ is multiplication by $\lambda_n$ ($n \neq 0$).

At the end of the day we have the following estimate for small $\ep$:
\[
|\lambda_l + \lambda_m - \lambda_n| >
\begin{cases}
1/2, & l + m \neq n,\\
\ep^2/5, & l + m = n,
\end{cases}
\]
see Proposition \ref{non-res}. Because this value appears in the denominator in (\ref{2nd-order}), if $g$ has size $\delta$, a direct application of the normal form transformation yields a lifespan comparable to $\ep^2/\delta^2$. To improve on this, we will make use of the structure of the nonlinearity:
\[
N(\mathfrak h, \mathfrak h) = \frac{1}{2}\mathfrak h_x^2 + O(|\ep|).
\]
The first term is the usual product-style nonlinearity, which imposes the restriction $l + \sgn l + m + \sgn m = n + \sgn n$, and implies that $l + m - n = \pm 1 \neq 0$, so the normal form transformation can be carried out as before. The second term is of size $|\ep|$ and gains a factor of $1/|\ep|$ in the lifespan. Thus the usual energy estimate can show a lifespan comparable to $1/|\ep\delta|$, and the normal form transformation can show a lifespan comparable to $|\ep|/\delta^2$. This decomposition of the nonlinearity into one part satisfying classical additive frequency restrictions, and another part enjoying better estimates analytically was first used in Germain--Pusateri--Rousset \cite{Germain-PR} to show global wellposedness of the 1D Schr\"odinger equation with potential (see also Chen--Pusateri \cite{Chen-P}). Our result shows that this approach can be adapted to quasilinear equations and to the case of discrete spectrum.

\subsection{The outline of the paper.} In section \ref{TW} we study the traveling waves solutions for $\eqref{BH}$. For sake of completeness we sketch the proof of existence which follows from bifurcation theory. In addition we  analyze the size of the traveling waves.  In section \ref{LTW} we study the linearization of equation \eqref{BH} around the traveling waves. In section \ref{NFR}, we introduce a new frame of references which will help us to avoid the resonances found in section \ref{LTW}. Finally, in section \ref{EE} we prove theorem \ref{mainthm}.

\section{Traveling waves}\label{TW}
The existence of traveling waves for \eqref{BH} was shown in \cite{Hunter}. Here we will study their size after we give some details about the existence proof. We look for solutions of \eqref{BH} in the form
\[
f_\ep(x, t) = u_\ep(x + v_\ep t),
\]
thus we have to find $(u_\ep,\, v_\ep)$ solving
\begin{align}\label{traveleq}
Hu_\ep - v_\ep u_\ep' + u_\ep u_\ep' = 0.
\end{align}
If $(u_\ep, v_\ep)$ is a solution, so is $(u_{\ep}^n(x), v_{\ep}^n) = (u_\ep(nx)/n, v_\ep/n)$. Thus from one solution we can get $n$-fold symmetric solutions for all $n \ge 1$.

To solve \eqref{traveleq}  we can apply the Crandall-Rabinowitz  theorem (see \cite{Crandall}) to
\begin{align*}
&F\,:\, \qquad (u, \mu)\qquad \quad\,\mapsto Hu + uu' - (-1 + \mu)u',\\
&\qquad H_r^{k,+}(\T) \times \C\qquad \to H_r^{k-1,-}(\T),
\end{align*}
where
\begin{align*}
&H^{k,+}_r(\T)\\
&=\left\{\text{$2\pi$-periodic, mean zero, and even functions, analytic in the strip $\{|\text{Im}(z)|<r\}$}\right\}
\end{align*}
endowed with the norm
\begin{align*}
||f||_{H^{k,+}_r(\T)}=\sum_{\pm} ||f(\cdot\pm i r)||_{H^k(\T)},
\end{align*}
 and
\begin{align*}
&H^{k,-}_r(\T)\\
&=\left\{\text{$2\pi$-periodic, and odd functions, analytic in the strip $\{|\text{Im}(z)|<r\}$}\right\}
\end{align*}
endowed with the norm
\begin{align*}
||f||_{H^{k,-}_r(\T)}=\sum_{\pm} ||f(\cdot\pm i r)||_{H^k(\T)}.
\end{align*}
Here
$||\cdot||_{H^k(\T)}$ is the usual Sobolev norm, and it is enough to take $k\geq 1$ and $r=1$.

We just notice that $F(0,\mu)=0$ and  the derivative of $F$ at $u=0$, $\mu=0$,
$$D_uF(0,0)h=Hh+h'$$
 just has a non trivial element in its kernel belonging to $H^{k,+}_r(\T)$, namely, $h=\cos(x)$.

 Thus, the  application of the C-R theorem allows to show the existence of a branch of  solutions $(u_\ep, v_\ep)\in (H^{1,+}_1,\R)$, bifurcating from $(0,-1)$, for \eqref{traveleq} with the assymptotic
\begin{align*}
 u_\ep(x)&=\ep \cos(x)+O(\ep^2)\\
 v_\ep&=-1+O(\ep).
\end{align*}

We remark that we obtain a bifurcation curve
\begin{align}
\ep\qquad \qquad & \to\qquad \quad (u_\ep, v_\ep)\nonumber\\
B_{\delta}=\{z\in \C\,:\, |z|<\delta\} &\to \qquad\left(H^{k-1,-}_r,\R\right)\label{analyticcurve}
\end{align}
which is differentiable and hence analytic on $B_\delta$ for $\delta$ small enough.

The rest of this section is devoted to proving further properties of these solutions.

Introducing the asymptotic expansion
\begin{align}\label{series}
u_\ep(x)&=\sum_{n=1}^\infty u_n(x)\ep^n, &
v_\ep&=\sum_{n=0}^\infty v_n \ep^n,
\end{align}
taking $u_1=\cos(x)$, $\lambda_0=-1$ and comparing the coefficient in $\ep^n$ we obtain that
\begin{align*}
u_n'+Hu_n=-v_{n-1}\sin(x)+\sum_{m=1}^{n-2}v_{m}u'_{n-m}-\frac{1}{2}\partial_x\sum_{m=1}^{n-1}u_{n-m}u_m=-v_{n-1}\sin(x)+f_n,
\end{align*}
for $n=2,3,...$

We notice that in order to solve the equation $Hu+u'=f$ we need $(f,\sin(x))=0$. Therefore we have to choose $v_{n-1}=\frac{1}{\pi}(\sin(x),f_n)$. This gives us a recurrence for $(u_n,v_{n-1})$, $n\geq 2$, in terms of $\{(u_m,v_{m-1})\}_{m=1}^{n-1}$. In order to study this recurrence we will introduce the ansatz
\begin{align}\label{subu}
u_n=\sum_{k=2}^n u_{n,k}\cos(kx).
\end{align}
By induction, one can check that the rest of coefficients in the expansion on cosines of $u_n$ must be zero. In addition, if $u_{\ep}(x)$ solves \eqref{traveleq},  $u_{-\ep}(x+\pi)$ is also a bifurcation curves in the direction of $\cos(x)$, and then by uniqueness, $u_{\ep}(x)=u_{-\ep}(x+\pi)$, which yields $u_{n,k}=0$ if $n-k=1$, mod$(2)$.

Comparing the coefficient of $\sin(kx)$, with $k=n$ mod$(2)$, and $2\leq k \leq n$, we have that
\begin{align}\label{coefu}
&(1-k)u_{n,k}+k\sum_{m=1}^{n-k} v_{m}u_{n-m,k}-\frac{k}{4}\sum_{m=1}^{n-1}\sum_{l=\max(1,k-n+m)}^{\min(m,k-1)}
u_{m,l}u_{n-m,k-l}\\&-\frac{k}{2}\sum_{m=1}^{n-1}\sum_{l=1}^{\min(m,n-m-k)}
u_{m,l}u_{n-m,k+l}=0\nonumber.
\end{align}
And comparing with $\sin(x)$ we have that
\begin{align}\label{coefv}
v_{n-1}=\frac{1}{2}\sum_{m=1}^{n-1}\sum_{l=1}^{\min(m,n-m-1)}u_{m,l}u_{n-m,1+l}.
\end{align}

Up to order $O(\ep^4)$ we find that

\begin{equation}\label{uv-Taylor-exp}
\begin{aligned}
u_\ep(x) &= \ep\cos x - \frac{\ep^2}{2}\cos2x + \frac{3\ep^3}{8}\cos3x + O(\ep^4),\\
v_\ep &= -1 - \frac{\ep^2}{4} + O(\ep^4).
\end{aligned}
\end{equation}

The  recurrence \eqref{coefu}-\eqref{coefv} allows us to prove the following result.

\begin{thm}The radius of convergence of the series \eqref{series}, with the coefficients given by the expression \eqref{subu}, \eqref{coefv} and \eqref{coefu} is not bigger than $\frac{2}{e}$.
\end{thm}
\begin{proof}
From \eqref{coefu} and \eqref{coefv} we have that
\[
(1 - n)u_{n,n} = \frac{1}{2}\sum_{k=1}^{n-1} (n - k)u_{k,k}u_{n-k,n-k}.
\]
Let
\[
y = y(x) = x + \sum_{n=2}^\infty u_{n,n}x^n.
\]
Then
\begin{align*}
y - xy' &= \frac{1}{2}yxy',\\
(2x + xy)y' &= 2y,\\
\frac{(2 + y)y'}{2y} &= \frac{1}{x},\\
\ln y + y/2 &= \ln x + C,\\
ye^{y/2} &= Cx.
\end{align*}
Since $y \sim x$ for small $x$, $C = 1$, so
\begin{align*}
ye^{y/2} &= x,\\
y &= 2W(x/2),
\end{align*}
where $W$ is the Lambert W-function. Since the radius of convergence of $W$ at 0 is $1/e$, the radius of convergence of $y$ at 0 is $2/e$, so the radius of convergence of \eqref{coefu} and \eqref{coefv} is at most $2/e$.
\end{proof}

In addition we can get a bound for how large the traveling wave can be.

\begin{thm}The series \eqref{series}, with the coefficients given by the expression \eqref{subu}, \eqref{coefv} and \eqref{coefu} converge for any $\ep<x^*\sim 0.23$.
\end{thm}
\begin{proof}
This proof is based on  the implicit funtion theorem.

Firsly we will introduce the spaces $L^{2,-}=\{ \text{odd functions }f\in L^2(\T)\}$, $H^{1,+}=\{\text{even functions } f\in H^1(\T)\}$. The space $X$ will the orthogonal complement of the span of $\cos(x)$ in $H^{1,+}$. We will equip $L^{2,-}$ with the norm \begin{align}\label{normal2}||u||_{L^{2,-}}^2=\frac{1}{\pi}\int_{-\pi}^\pi |u(x)|^2 dx,\end{align} in such a way that $||\sin(nx)||_{L^{2,-}}=1$, for $n\geq 1$. We also define \begin{align}\label{normaX}||u||_{X}^2=\frac{1}{\pi}\int_{-\pi}^\pi \left(|u'(x)|^2+|u(x)|^2-2 u(x)\Lambda u(x)\right)dx.\end{align} Thus $||\cos(nx)||_{X}=n-1$, for $n\geq 2$. The reason why we take these norms is technical and it will arise below. Finally we define
$$\mathcal{X}=X\times \R$$
equipped with the norm $$||(\tilde{u},\nu)||_{\mathcal{X}}=\sqrt{||\tilde{u}||_{X}^2+|\nu|^2}.$$

Since $u_\ep = \ep\cos x - \frac{1}{2}\ep^2\cos2x + O(\ep^3)$ and $v_\ep = -1 + O(\ep^2)$, we can let
\begin{align*}
G(\ep, \tilde u, \mu) &= \frac{1}{\ep^2}F\left(\ep\cos x - \frac{\ep^2}{2}\cos2x + \ep^2\tilde u, \ep\mu \right)\\
&= \frac{1}{\ep}\left( \sin x - \frac{\ep}{2}\sin2x + \ep H\tilde u \right)\\
&- \left( \cos x - \frac{\ep}{2}\cos2x + \ep\tilde u \right)(\sin x - \ep\sin2x - \ep\tilde u')\\
&- \frac{1}{\ep}(1 + \ep\mu)(\sin x - \ep\sin2x - \ep\tilde u')\\
&= H\tilde u + \ep\Big(\cos x(\sin2x + \tilde u')\\
&+ \left. \left( \frac{1}{2}\cos2x - \tilde u \right)(\sin x - \ep\sin2x - \ep\tilde u') \right)\\
&+ \tilde u' - \mu(\sin x - \ep\sin2x - \ep\tilde u')
\end{align*}
map $\R \times \mathcal{X}$ to $L^{2,-}$.

Because the existence of traveling waves we already know that there exists $\ep^*$ such that for every $\ep\in [0,\ep^*)$, there exist $\tilde{u}_\ep$ and $\mu_\ep$ satisfying
$$G(\ep,\tilde{u}_\ep,\mu_\ep)=0.$$
In addition we have that
\begin{align*}
& \left.\frac{dG(\ep, \tilde{u}_\ep+s \tilde{v},\mu+s \nu)}{ds}\right|_{s=0}\equiv dG_{\ep,\tilde u_\ep,\mu}(\tilde v, \nu)\\&= H\tilde v + \ep\Big( \tilde v'\cos x - \tilde v(\sin x -\ep\sin2x - \ep\tilde u')- \left. \ep\left( \frac{1}{2}\cos2x - \tilde u \right)\tilde v' \right)\\
&+\tilde v' - \nu(\sin x-\ep\sin2x - \ep\tilde u') + \ep\mu\tilde v'
\end{align*}
maps $(\tilde v, \nu) \in \mathcal{X}$ linearly to $L^{2,-}$.

Thus as far as $dG_{\ep,\tilde{u}_\ep, \mu_\ep}(\tilde{u},\mu)$ is invertible from $\mathcal{X}$ to $L^{2,-}$ for $\ep\in [0,x^*)$ we will be able to extend the solution $(u_\ep,\mu_\ep)$ from $[0,\ep^*)$ to $[0,x^*)$ by the Implicit Function Theorem.

 Note that
\[
dG_{0,0,0}(\tilde v, \nu) = H\tilde v + \tilde v' - \nu\sin x
\]
is an isometry from $\mathcal{X}$ to $L^{2,-}$ under the norms given by \eqref{normal2} and \eqref{normaX}. Therefore one can compute
\begin{align*}
dG_{\ep,\tilde{u}_\ep,\mu_\ep}=dG_{0,0,0}^{-1}\left(\mathbb{I}+dG_{0,0,0}^{-1}\left(dG_{\ep,\tilde{u}_\ep,\mu_\ep}-dG_{0,0,0}\right)\right).
\end{align*}

By the Neumann series and the fact that $dG_{0,0,0}$ is an isometry, $dG_{\ep,\tilde{u}_\ep,\mu_\ep}$ will be invertible, as far as, $||dG_{\ep,\tilde{u}_\ep,\mu_\ep}-dG_{0,0,0}||_{\mathcal{X}\to L^{2,-}}<1$. In order to show this last inequality we will bound $$||dG_{\ep,\tilde{u}_\ep,\mu_\ep}-dG_{0,0,0}||_{\mathcal{X}\to L^{2,-}}\equiv A_\ep $$ in terms of $||\tilde{u}_\ep||_{X}$ and $\mu_\ep$. After that we will bound $||\tilde{u}_\ep||_{X}$ and $\mu_\ep$.  To do it we will use the information we have about $\partial_\ep \tilde{u}_\ep$ and $\partial_\ep\mu_\ep$.

 Along the bifurcation curve,
\begin{equation}\label{implicit-diff2}
\begin{aligned}
dG_{\ep,\tilde u_\ep,\mu_\ep}(\partial_\ep\tilde u_\ep, \mu_\ep')
&= - \partial_\ep G(\ep, \tilde u_\ep, \mu_\ep)\\
&= \cos x(\sin2x + \tilde u_\ep') + \frac{\sin x}{2}(\cos2x - 2\tilde u_\ep)\\
&- \ep(\cos2x - 2\tilde u_\ep)(\sin2x + \tilde u_\ep') + \mu_\ep(\sin2x + \tilde u_\ep').
\end{aligned}
\end{equation}
Thus
$$(\partial_\ep \tilde{u}_\ep,\mu'_\ep)=dG_{\ep,\tilde{u}_\ep,\mu_\ep}^{-1}\left(-\partial_\ep G(\ep,\tilde{u}_\ep,\mu_\ep)\right)$$

Therefore
\begin{align}\label{thislast}
\sqrt{||\partial_\ep \tilde{u}_\ep||_{X}^2+|\mu'_\ep|^2}\leq \frac{1}{1-A_\ep}||\partial_\ep G(\ep,\tilde{u}_\ep,\mu_\ep)||_{L^{2,-}}.
\end{align}
In addition we have that, for $r_\ep=\sqrt{|| \tilde{u}_\ep||_{X}^2+|\mu_\ep|^2}$,
$$\partial_\ep r_\ep \leq \sqrt{||\partial_\ep \tilde{u}_\ep||_{X}^2+|\mu'_\ep|^2}\leq \frac{1}{1-A_\ep}||\partial_\ep G(\ep,\tilde{u}_\ep,\mu_\ep)||_{L^{2,-}}.$$

Thus, explicit estimates for $A_\ep$ and $||\partial_\ep G(\ep,\tilde{u}_\ep,\mu_\ep)||_{L^{2,-}}$ in terms of $r_\ep$ and $\ep$ give a differential inequality for $r_\ep$ which can be used to bound $A_\ep$.

We will need the following lemmas to bound $A_\ep$ and  the norm $||\partial_\ep G(\ep,\tilde{u}_\ep,\mu_\ep)||_{L^{2,-}}$, where $\partial_\ep G(\ep,\tilde{u}_\ep,\mu_\ep)$ is given by the right hand side of  \eqref{implicit-diff2}.

\begin{lem}\label{linear-bound}
If $f \in X$ then $\|f\sin x - f'\cos x\|_{L^2} \le \sqrt{3}\|f\|_X$.
\end{lem}
\begin{proof}
Let
\[
f = \sum_{n=2}^\infty f_n\cos nx.
\]
Then
\begin{align*}
f\sin x - f'\cos x &= -(f\cos x)'\\
&= -\frac{1}{2}\left( f_2\cos x + f_3\cos2x + \sum_{n=3}^\infty (f_{n-1} + f_{n+1})\cos nx \right)'\\
&= \frac{1}{2}\left( f_2\sin x + 2f_3\sin2x + \sum_{n=3}^\infty n(f_{n-1} + f_{n+1})\sin nx \right).
\end{align*}
Then
\begin{align*}
||f\sin(x)-f'\cos(x)||_{L^2}^2=\frac{1}{2}\sqrt{f_2^2+4f_3^2+\sum_{n=3}^\infty n^2(f_{n-1}+f_{n+1})^2}
\end{align*}
And we have that
\begin{align*}
\sum_{n=3}^\infty n^2(f_{n-1}+f_{n+1})^2=\sum_{n=3}^\infty n^2f_{n-1}^2+\sum_{n=3}^\infty f_{n+1}^2+2\sum_{n=2}^\infty n^2f_{n-1}f_{n+1}.
\end{align*}
In addition
\begin{align*}
\sum_{n=3}^\infty n^2 f_{n-1}^2=\sum_{n=2}^\infty (n+1)^2f_{n}^2=9f_2^2+16f_3^2+25f_4^2+36f_5^2+\sum_{n=6}^\infty (n+1)^2 f_n^2,
\end{align*}
\begin{align*}
\sum_{n=3}^\infty n^2 f_{n+1}^2 =\sum_{n=4}^\infty (n-1)^2 f_{n}^2= 9f_4^2+16f_5^2+\sum_{n=6}^\infty (n-1)^2 f_{n}^2,
\end{align*}
and
\begin{align*}
&2\sum_{n=3}^\infty n^2 f_{n-1}f_{n+1}=2\cdot 9f_2f_4+2\cdot 16 f_{3}f_5+2\sum_{n=5}^\infty n^2 f_{n-1}f_{n+1}\\
&\leq 9\sigma f_2^2+\frac{9}{\sigma}f_4^2 + 16\gamma f_3^2+\frac{16}{\gamma}f_5^2+\sum_{n=5}^\infty n^2f_{n-1}^2+\sum_{n=5}^\infty n^2 f_{n+1}^2\\
&=9\sigma f_2^2+\left(\frac{9}{\sigma}+25\right)f_4^2 + 16\gamma f_3^2+\left(\frac{16}{\gamma}+36\right)f_5^2+\sum_{n=7}^\infty n^2f_{n-1}^2+\sum_{n=5}^\infty n^2 f_{n+1}^2\\
&=9\sigma f_2^2+\left(\frac{9}{\sigma}+25\right)f_4^2 + 16\gamma f_3^2+\left(\frac{16}{\gamma}+36\right)f_5^2+\sum_{n=6}^\infty (n+1)^2f_{n}^2+\sum_{n=6}^\infty (n-1)^2 f_{n}^2\\
\end{align*}
Thus
\begin{align*}
&f_2^2+4 f_3^2+\sum_{n=3}^\infty n^2 (f_{n-1}+f_{n+1})\\
&\leq \left(10+9\sigma\right)f_2^2+(20+16\gamma)f_{3}^2+\left(59+\frac{9}{\sigma}\right)f_4^2+\left(88+\frac{16}{\gamma}\right)f_5^2\\&+ 2\sum_{n=6}^\infty \left(1+\frac{(n+1)^2}{(n-1)^2}\right)(n-1)^2f_{n}^2.
\end{align*}
Since, for $n\geq 6$, we have that $\frac{(n+1)^2}{(n-1)^2}\leq \frac{49}{25}$ we finally obtain that
\begin{align*}
&f_2^2+4 f_3^2+\sum_{n=3}^\infty n^2 (f_{n-1}+f_{n+1})\\
&\leq \left(10+9\sigma\right)f_2^2+(20+16\gamma)f_{3}^2+\left(59+\frac{9}{\sigma}\right)f_4^2+\left(88+\frac{16}{\gamma}\right)f_5^2\\&+ 2\left(\frac{49}{25}+1\right)\sum_{n=6}^\infty (n-1)^2f_{n}^2.
\end{align*}
We have to compare
\begin{align*}
10+9\sigma, && \frac{20+16\gamma}{4}, && \frac{59+\frac{9}{\sigma}}{9}, && \frac{88+\frac{16}{\gamma}}{16}&&\text{and}&& 2\left(1+\frac{49}{25}\right)
\end{align*}
The minimum in $\sigma$ of $\max\left(10+9\sigma, \frac{59+\frac{9}{\sigma}}{9}\right)$ is attached when $10+9\sigma=\frac{59}{9}+\frac{1}{\sigma}$ and it is 11.373. Then minimum in $\gamma$ of $\max \left(5+4\gamma, 5.5+\frac{1}{\gamma}\right)$ less than 7.5. Finally $2(1+49/25)=5.92$. Therefore
\begin{align*}
\|f\sin(x)-f'\cos(x)\|_{L^2}\leq \sqrt{3}\|f\|_X.
\end{align*}

\end{proof}

\begin{lem}\label{bilinear-bound}
If $f$, $g \in X$ then $\|(fg)'\|_{L^2} \le B\|f\|_X\|g\|_X$, where
\[
B = \sqrt{\frac{\pi^2}{3} + \frac{869}{144}} \approx 3.05.
\]
\end{lem}
\begin{proof}
Let
\begin{align*}
f &= \sum_{n=2}^\infty f_n\cos nx = \frac{1}{2} \sum_{|n|\ge2} f_{|n|}e^{inx}, &
g &= \sum_{n=2}^\infty g_n\cos nx = \frac{1}{2} \sum_{|n|\ge2} g_{|n|}e^{inx} \in X.
\end{align*}
Then
\begin{align*}
(fg)' &= \frac{i}{4}\sum_{|n|\ge1} n\sum_{|m|\ge2,|n-m|\ge2} f_{|m|}g_{|n-m|}e^{inx}\\
&= -\frac{1}{2}\sum_{n\ge1} n\sum_{|m|\ge2,|n-m|\ge2} f_{|m|}g_{|n-m|}\sin nx
\end{align*}
so by Cauchy--Schwarz,
\begin{align*}
\|(fg)'\|_{L^2}^2 &= \frac{1}{8}\sum_{|n|\ge1} n^2\left( \sum_{|m|\ge2,|n-m|\ge2} f_{|m|}g_{|n-m|} \right)^2\\
&\le \frac{1}{8} \sum_{|n|\ge1} \sum_{|m|\ge2,|n-m|\ge2} \frac{n^2}{(|m|-1)^2(|n-m|-1)^2}\\
&\times \sum_{|m|\ge2,|n-m|\ge2} (|m| - 1)^2f_m^2(|n - m| - 1)^2g_{n-m}^2\\
&\le \frac{C}{8} \sum_{|n|\ge1}\sum_{|m|\ge2,|n-m|\ge2} (|m| - 1)^2f_m^2(|n - m| - 1)^2g_{n-m}^2\\
&= \frac{C}{2}\|f\|_X\|g\|_X
\end{align*}
where
\begin{align*}
C &= \sup_{n=1}^\infty C_n,\\
C_n &= \sum_{|m|\ge2,|n-m|\ge2} \frac{n^2}{(|m|-1)^2(|n-m|-1)^2}\\
&= \underbrace{\sum_{k=1}^{n-3} \frac{n^2}{k^2(n - k - 2)^2}}_{D_n} + 2\underbrace{\sum_{k=1}^\infty \frac{n^2}{k^2(k + n)^2}}_{E_n}.
\end{align*}
$D_n \neq 0$ only when $n \ge 4$. We have that $D_4 = 16$ and when $n\ge5$,
\begin{align*}
D_n &\le \frac{2n^2}{(n - 2)^2}\sum_{k=1}^{n-3} \frac{k^2 + (n - k - 2)^2}{k^2(n - k - 2)^2} = \frac{4n^2}{(n - 2)^2}\sum_{k=1}^{n-3} \frac{1}{k^2}\\
&\le \frac{4n^2}{(n - 2)^2}\left( \frac{\pi^2}{6} - \frac{1}{n - 2} \right)
\le \frac{50(\pi^2 - 2)}{27}.
\end{align*}
For $E_n$, by partial fraction decomposition
\[
\frac{n^2}{k^2(k + n)^2} = \frac{1}{k^2} - \frac{2}{kn} + \frac{2}{(k + n)n} + \frac{1}{(k + n)^2}
\]
so
\begin{align*}
E_n &:= \sum_{k=1}^\infty \frac{n^2}{k^2(k + n)^2} = \frac{\pi^2}{6} + \sum_{k=n+1}^\infty \frac{1}{k^2} - \frac{2}{n}\sum_{k=1}^n \frac{1}{k}
< \frac{\pi^2}{3} - \sum_{k=1}^n \frac{1}{k^2}
\end{align*}
so when $n \ge 5$, $E_n < \frac{1}{3}(\pi^2 - 4)$ so
\[
C_n = D_n + 2E_n < \frac{68\pi^2 - 172}{27} < 18.5
\]
and when $1 \le n \le 3$,
\[
C_n = 2E_n < \frac{2\pi^2}{3} - 2 < 5.
\]
For $n = 4$ we have
\[
C_4 = D_4 + 2E_4 = 16 + 2\left( \frac{\pi^2}{3} - \frac{205}{144} \right) - \frac{13}{12} = \frac{2\pi^2}{3} + \frac{869}{72} > 18.6
\]
so
\[
C = \sup_{n=1}^\infty C_n = C_4 = \frac{2\pi^2}{3} + \frac{869}{72} = 2B^2
\]
and finally
\[
\|(fg)'\|_{L^2} \le \sqrt{\frac{C}{2}}\|f\|_{H^1}^2\|g\|_{H^1}^2
= B\|f\|_{H^1}^2\|g\|_{H^1}^2.
\]
\end{proof}

\begin{lem}\label{linear-bound2}
If $f \in X$ then $\|2f\sin2x - f'\cos2x\|_{L^2} \le 0.5\sqrt{17}\|f\|_X$.
\end{lem}
\begin{proof}
Let
\[
f = \sum_{n=2}^\infty f_n\cos nx.
\]
Then
\begin{align*}
2f\sin2x - f'\cos2x &= -(f\cos2x)'
= -\frac{1}{2}(f_3\cos x + f_4\cos2x + f_5\cos3x)'\\
&-\frac{1}{2}\left( \sum_{n=4}^\infty (f_{n-2} + f_{n+2})\cos nx \right)'\\
&= \frac{1}{2}(f_3\sin x + 2f_4\sin2x + 3f_5\sin3x)\\
&+ \frac{1}{2}\left( \sum_{n=4}^\infty n(f_{n-2} + f_{n+2})\sin nx \right),
\end{align*}
and
\begin{align*}
||2f\sin2x - f'\cos2x &= -(f\cos2x)'||_{L^2}=\frac{1}{2}\sqrt{f_3^2+4f_4^2+9f_5^2+\sum_{n=4}^\infty n^2(f_{n-2}+f_{n+2})^2}.
\end{align*}
We have that
\begin{align*}
&\sum_{n=4}^\infty n^2(f_{n-2}+f_{n+2})^2=\sum_{n=4}^\infty n^2f_{n-2}^2+\sum_{n=4}^\infty n^2 f_{n+2}^2+2\sum_{n=4}^\infty n^2f_{n-2}f_{n+2}\\
&=\sum_{n=2}^\infty (n+2)^2f_{n}^2+\sum_{n=6}^\infty (n-2)^2 f_{n}^2 +2\sum_{n=4}^\infty n^2f_{n-2}f_{n+2}
\end{align*}
We also can bound
\begin{align*}
&2\sum_{n=4}^\infty n^2 f_{n-2}f_{n+2}\leq \sum_{n=4}^\infty n^2\sigma_{n-2}f_{n-2}^2+\sum_{n=4}^\infty n^2\frac{1}{\sigma_{n-2}}f_{n+2}^2\\
&=\sum_{n=2}^\infty (n+2)^2\sigma_n f_n^2+\sum_{n=6}^\infty (n-2)^2\sigma_{n-4}f_{n}^2.
\end{align*}
Therefore
\begin{align*}
&f_3^2+4f_4^2+9f_5^2+\sum_{n=4}^\infty n^2(f_{n-2}+f_{n+2})^2\\ &\leq
f_3^2+4f_4^2+9f_5^2+\sum_{n=2}^\infty (n+2)^2(1+\sigma_n) f_{n}^2+\sum_{n=6}^\infty (n-2)^2\left(1+\frac{1}{\sigma_{n-4}}\right)f_{n}^2\\
&= 16(1+\sigma_2)f_2^2+ (1+25(1+\sigma_3))f_3^2+(4+36(1+\sigma_3))f_4^2+(9+49(1+\sigma_5))f_5^2\\
&+\sum_{n=6}^\infty \left((n+2)^2(1+\sigma_n)+(n-2)^2\left(1+\frac{1}{\sigma_{n-4}}\right)\right)f_{n}^2\\
&= 16(1+\sigma_2)f_2^2+ (1+25(1+\sigma_3))f_3^2+(4+36(1+\sigma_3))f_4^2+(9+49(1+\sigma_5))f_5^2\\
&+\left(64(1+\sigma_6)+16\left(1+\frac{1}{\sigma_2}\right)\right)f_6^2+\left(81(1+\sigma_7)+25\left(1+\frac{1}{\sigma_3}\right)\right)f_7^2\\
&+\left(100(1+\sigma_8)+36\left(1+\frac{1}{\sigma_4}\right)\right)f_8^2+\left(121(1+\sigma_9)+49\left(1+\frac{1}{\sigma_5}\right)\right)f_9^2\\
&+\sum_{n=10}^\infty \left((n+2)^2(1+\sigma_n)+(n-2)^2\left(1+\frac{1}{\sigma_{n-4}}\right)\right)f_{n}^2.
\end{align*}
We take $\sigma_n=1$ for $n\geq 6$ to get
\begin{align*}
&f_3^2+4f_4^2+9f_5^2+\sum_{n=4}^\infty n^2(f_{n-2}+f_{n+2})^2\\ &\leq
16(1+\sigma_2)f_2^2+ (1+25(1+\sigma_3))f_3^2+(4+36(1+\sigma_3))f_4^2+(9+49(1+\sigma_5))f_5^2\\
&+\left(128+16\left(1+\frac{1}{\sigma_2}\right)\right)f_6^2+\left(162+25\left(1+\frac{1}{\sigma_3}\right)\right)f_7^2\\
&+\left(200+36\left(1+\frac{1}{\sigma_4}\right)\right)f_8^2+\left(242+49\left(1+\frac{1}{\sigma_5}\right)\right)f_9^2\\
&+\sum_{n=10}^\infty \left((n+2)^2+(n-2)^2\right)f_{n}^2.
\end{align*}
One can compute that $(n+2)^2+(n-2)^2=2(n^2+4)$ and that $$2\frac{n^2+4}{(n-1)^2}\leq 2\frac{104}{81} .$$

Thus we have to compare
\begin{align*}
\max_{\sigma_2\in [0,1]} \left(16(1+\sigma_2), \frac{144}{25}+\frac{16}{25\sigma_2}\right)
\end{align*}
and
\begin{align*}
\max_{\sigma_3\in [0,1]} \left(\frac{26}{4}+\frac{25}{4}\sigma_3, \frac{187}{36}+\frac{25}{36\sigma_3} \right)
\end{align*}
and
\begin{align*}
\max_{\sigma_4\in [0,1]} \left(\frac{40}{9}+\frac{36}{9}\sigma_4, \frac{236}{49}+\frac{36}{49\sigma_4} \right)
\end{align*}
and
\begin{align*}
\max_{\sigma_5\in [0,1]} \left(\frac{58}{16}+\frac{49}{16}\sigma_5, \frac{291}{64}+\frac{49}{64\sigma_5} \right).
\end{align*}
and
$$2\frac{104}{81}<2.6.$$
Then the maximum of these all numbers is the first one which is $\leq 17$. Thus
\begin{align*}
\|2f\sin2x - f'\cos2x\|_{L^2} \le \frac{\sqrt{17}}{2}\|f\|_X
\end{align*}

\end{proof}

Now, with the lemmas \ref{linear-bound}, \ref{bilinear-bound} and \ref{linear-bound2} we are ready to bound the right hand side of \eqref{implicit-diff2}. Indeed,
\begin{align*}
\|\cos x\sin2x + \frac{1}{2}\sin x\cos2x - \ep\cos2x\sin2x\|_{L^2}
= &\frac{1}{4}\|3\sin3x + \sin x - 2\ep\sin4x\|_{L^2}\\  = & \frac{\sqrt{10 + 4\ep^2}}{4},\\
\|\tilde u_\ep'\cos x - \tilde u_\ep\sin x\|_{L^2} \le &\sqrt{3}\|\tilde u_\ep\|_X, \tag{Lemma \ref{linear-bound}}\\
\|2\tilde u_\ep\sin2x - \tilde u_\ep'\cos2x\|_{L^2} \le & 0.5\sqrt{17}\|\tilde u_\ep\|_X, \tag{Lemma \ref{linear-bound2}}\\
\|2\tilde u_\ep\tilde u_\ep'\| = \|((\tilde u_\ep)^2)'\|_{L^2} \le & B\|\tilde u_\ep\|_X^2, \tag{Lemma \ref{bilinear-bound}}\\
\|\mu_\ep\sin2x\|_{L^2} \le &|\mu_\ep|\\
\|\mu_\ep\tilde u_\ep'\|_{L^2} \le & 2|\mu_\ep|\|\tilde u_\ep\|_X,
\end{align*}
by Cauchy--Schwarz,
\begin{align*}
\|\text{right-hand side of (\ref{implicit-diff2})}\|_{L^2}
&\le \frac{\sqrt{10 + 4\ep^2}}{4} + (\sqrt{3} + 0.5\sqrt{17}\ep)\|\tilde u_\ep\|_X + B\ep\|\tilde u_\ep\|_X^2\\
&+ |\mu_\ep| + 2|\mu_\ep|\|\tilde u_\ep\|_X\\
&\le \frac{\sqrt{10 + 4\ep^2}}{4} + 2\sqrt{\|\tilde u_\ep\|_X^2 + |\mu_\ep|^2} + 0.5\sqrt{17}\ep\|\tilde u_\ep\|_X\\
&+ B\ep\|\tilde u_\ep\|_X^2 + \|\tilde u_\ep\|_X^2 + |\mu_\ep|^2.
\end{align*}
Turning to the other side, we have that
\begin{equation}\label{dG-dG2}
\begin{aligned}
(dG_{(\ep,\tilde u_\ep,\mu_\ep)} - dG_{(0,0,0)})(\tilde v, \nu)
&= \ep\Big(\tilde v'\cos x - \tilde v(\sin x - \ep\sin2x - \ep\tilde u_\ep')\\
&- \left. \ep\left( \frac{1}{2}\cos2x - \tilde u_\ep \right)\tilde v' \right)\\
&+ \ep\nu(\sin2x + \tilde u_\ep') + \ep\mu_\ep\tilde v'.
\end{aligned}
\end{equation}
Since
\begin{align*}
\|\tilde v'\cos x - \tilde v\sin x\|_{L^2} &\le \sqrt{3}\|\tilde v\|_X,
\tag{Lemma \ref{linear-bound}}\\
\|\tilde v\sin2x - \frac{1}{2}\tilde v'\cos2x\|_{L^2} &\le 0.25\sqrt{17}\|\tilde v\|_X, \tag{Lemma \ref{linear-bound2}}\\
\|\tilde v\tilde u_\ep' + \tilde u_\ep\tilde v'\|_{L^2} &= \|(\tilde u_\ep\tilde v)'\|_{L^2} \le B\|\tilde u_\ep\|_X\|\tilde v\|_X,\tag{Lemma \ref{bilinear-bound}}\\
\|\nu(\sin2x + \tilde u_\ep')\|_{L^2} &\le |\nu|(1 + 2\|\tilde u_\ep\|_X),\\
\|\mu_\ep\tilde v'\|_{L^2} &\le 2|\mu_\ep|\|\tilde v\|_X,
\end{align*}
\begin{align*}
\|\text{left-hand side of (\ref{dG-dG2})}\|_{L^2} &\le (\sqrt3\ep + 0.25\sqrt{17}\ep^2 + B\ep^2\|\tilde u_\ep\|_X + 2\ep|\mu_\ep|)\|\tilde v\|_X\\
&+ \ep(1 + 2\|\tilde u_\ep\|_X)|\nu|
\end{align*}
so by the Minkowski inequality,
\begin{align*}
&\|dG_{(\ep,\tilde u_\ep,\mu_\ep)} - dG_{(0,0,0)}\|_{X\times\R\to L^2}\\
\le &\sqrt{(\sqrt{3}\ep + 2\ep|\mu_\ep| + 0.25\sqrt{17}\ep^2 + B\ep^2\|\tilde u_\ep\|_X)^2 + \ep^2(1 + 2\|\tilde u_\ep\|_X)^2}\\
\le &\sqrt{(3\ep^2 + \ep^2)} + 2\ep\sqrt{|\mu_\ep|^2 + \|\tilde u_\ep\|_X^2} + 0.25\sqrt{17}\ep^2 + B\ep^2\|\tilde u_\ep\|_X\\
\leq & 2\ep + 2\ep\sqrt{|\mu_\ep|^2 + \|\tilde u_\ep\|_X^2} + 0.25\sqrt{17}\ep^2 + B\ep^2\|\tilde u_\ep\|_X =: A_\ep.
\end{align*}
Since $dG_{(0,0,0)}$ is an isometry, the Von Neumann series $(1 - T)^{-1} = \sum_{n=0}^\infty T^n$ shows that if $A_\ep < 1$, then $dG_{(\ep,\tilde u_\ep,\mu_\ep)}$ is invertible, and $\|dG_{(\ep,\tilde u_\ep,\mu_\ep)}^{-1}\| \le (1 - A_\ep)^{-1}$, so
\begin{align*}\label{r'-bound2}
\sqrt{\|\partial_\ep\tilde u_\ep\|_X^2 + |\mu_\ep'|^2}
&\le \frac{1}{1 - A_\ep}\left( \frac{\sqrt{10 + 4\ep^2}}{4} + 2\sqrt{\|\tilde u_\ep\|_X^2 + |\mu_\ep|^2} + \frac{\sqrt{17}\ep}{2}\|\tilde u_\ep\|_X \right.\\
&+ B\ep\|\tilde u_\ep\|_X^2 + \|\tilde u_\ep\|_X^2 + |\mu_\ep|^2 \Big).
\end{align*}
Let $r_\ep = \sqrt{\|\tilde u_\ep\|_X^2 + |\mu_\ep|^2}$. Then
\[
A_\ep \le 2\ep + 2\ep r_\ep + 0.25\sqrt{17}\ep^2 + B\ep^2r_\ep,
\]
$r_0 = 0$ and
\[
r_\ep' \le \frac{\frac{1}{4}\sqrt{10 + 4\ep^2} + (2 + \frac{1}{2}\sqrt{17}\ep)r_\ep + B\ep r_\ep^2 + r_\ep^2}{1 - 2\ep - 2\ep r_\ep - \frac{1}{4}\sqrt{17}\ep^2 - B\ep^2r_\ep}.
\]
By the comparison principle, $r_\ep$ is bounded from above by the solution to
\[\label{ode}
\frac{dy}{dx} = y' = \frac{\sqrt{10 + 4x^2} + (8 + 2\sqrt{17}x)y + 4Bxy^2 + 4y^2}{4 - 8x - 8xy - \sqrt{17}x^2 - 4Bx^2y}
\]
with $y(0) = 0$. Integrating
\begin{align*}
\sqrt{10 + 4x^2}dx + (8 +2\sqrt{17}x)ydx + 4Bxy^2dx + 4y^2dx\\
+ 8xdy + 8xydy + \sqrt{17}x^2dy + 4Bx^2ydy = 4dy
\end{align*}
gives
\[
x\sqrt{x^2 + 2.5} + 2.5\sinh^{-1}(\sqrt{0.4}x)
+ 8xy + \sqrt{17}x^2y + 2Bx^2y^2 + 4xy^2 = 4y + c.
\]
Since $y(0) = 0$, $c = 0$, so
\[
(2Bx^2 + 4x)y^2 + (8x + \sqrt{17}x^2 - 4)y + x\sqrt{x^2 + 2.5} + 2.5\sinh^{-1}(\sqrt{0.4}x) = 0.
\]
When $x > 0$, the quadratic coefficient and the constant is positive,
so this equation has a non-negative root iff
\[
8x + \sqrt{17}x^2 - 4 \le -2\sqrt{(2Bx^2 + 4x)(x\sqrt{x^2 + 2.5} + 2.5\sinh^{-1}(\sqrt{0.4}x))}
\]
whose solution is $x \le x^* \approx 0.23$ numerically.
Hence the solution can be extended to $\ep = x^* \approx 0.23$. In order to achieve this last conclusion we just notice that the solution to  \eqref{ode}, with $y(0)=0$ can be extended only if $A_\ep<1$, since $1-A_\ep$ arises in the denominator.

The above argument shows that for $\ep \in (-x^*, x^*)$, the bifurcation curve produces a traveling wave $u_\ep = \ep\cos x - \frac{\ep^2}{2}\cos2x + \ep^2\tilde u_\ep$ which travels at speed $v_\ep = -1 - \ep\mu_\ep$.
Since all the operators involved are analytic in all its arguments,
the bifurcation curve is analytic in $\ep$ on $(-x^*, x^*)$.
It may be the case, however, that the power series for $u_\ep$ and $v_\ep$ around $\ep = 0$ has a smaller radius of convergence than $x^*$
(for example, the function $f(x) = (x^2 + 1)^{-1}$ is analytic on the whole real line, but the radius of convergence of its power series around 0 is only 1.) We now show that the radius of convergence of the power series for $u_\ep$ and $v_\ep$ are indeed at least $x^*$.

We note that the above argument also works if $\ep$ is replaced with $\ep e^{ia}$ ($a \in \R$), so the bifurcation curve $(u_\ep, v_\ep)$ is also analytic in a neighborhood of $\{\ep e^{ia}: \ep \in (-x^*, x^*)\}$.
Hence the curve is analytic in the disk of radius $x^*$ centered at 0,
so the radius of convergence of its power series around 0 is at least $x^*$.
\end{proof}

\section{Linearization around traveling waves}\label{LTW}

In this section we will analyse the spectrum of the operator
\begin{align*}
L_\ep g =-v_\ep g_x+Hg+\left(u_\ep(x)g\right)_x
\end{align*}
corresponding to the linearization of equation \eqref{BH} around the traveling wave $(u_\ep, v_\ep)$ bifurcating from zero in the direction of the cosine studied in the previous section.

Actually, let
\[
f(x, t) = f_\ep(x, t) + g(x + v_\ep t, t).
\]
with $f_\ep(x,t)=u_\ep(x+v_\ep t)$.
Then
\[
f_t(x, t) = \partial_tf_\ep(x, t) + (v_\ep g_x + g_t)(x + v_\ep t, t)
\]
and
\begin{align*}
(Hf + ff_x)(x, t) &= (Hf_\ep + f_\ep\partial_xf_\ep)(x, t) + Hg(x + v_\ep t, t)\\
&+ \partial_x(f_\ep(x, t)g(x + v_\ep t, t))
+ g(x + v_\ep t, t)\partial_xg(x + v_\ep t, t).
\end{align*}
Putting these in \eqref{BH}, we get the equation for $g(x,t)$:
\[
\partial_t g(x,t) = -v_\ep g(x,t)_x + Hg(x,t) + (u_\ep(x) g(x,t))_x + g(x,t)g(x,t)_x.
\]
The linearization around $g = 0$ is
\begin{align*}
 \partial_t g= L_\ep g
\end{align*}
where
\begin{equation}\label{L-ep}
L_\ep g = -v_\ep g_x + Hg + (u_\ep g)_x
= \underbrace{Hg + g_x}_{Lg} + \sum_{n=1}^\infty \ep^n\underbrace{((u^{(n)} - v^{(n)})g)_x}_{L^{(n)}g}.
\end{equation}

\subsection{The eigenvalue 0}\label{lambda=0}
The action of $L$ on the Fourier modes is
\[
\mathcal F(Lg)(m) = i(m - \sgn m)\hat g(m)
\]
with eigenvalues 0 (double), $\pm i$, $\pm2i$, $\dots$ (on $L^2(\T)$ with zero mean). We first study the perturbation of the eigenspace corresponding to the double eigenvalue of 0.
By translational symmetry, for any $\dl \in \R$, $u_\ep(x + \dl)$ is also a solution to
\[
Hu - v_\ep u + uu' = 0.
\]
Differentiation with respect to $\dl$ then shows that
\[
L_\ep u_\ep' = Hu_\ep' - v_\ep u_\ep' + (u_\ep u_\ep')' = 0.
\]
Also, since $u_\ep$ lies on a bifucation curve, we can differentiate
\[
Hu_\ep - v_\ep u_\ep' + u_\ep u_\ep' = 0
\]
with respect to $\ep$ to get
\[
L_\ep\partial_\ep u_\ep = H\partial_\ep u_\ep - (\partial_\ep v_\ep)u_\ep' + u_\ep\partial_\ep u_\ep' + u_\ep'\partial_\ep u_\ep = (\partial_\ep v_\ep)u_\ep'
\]
so on the span $V_\ep$ of $u_\ep'$ and $\partial_\ep u_\ep$, $L_\ep$ acts nilpotently by the matrix
\[
\begin{pmatrix}
0 & \partial_\ep v_\ep\\
0 & 0
\end{pmatrix}.
\]

\subsection{Simplifying the linearized operator}\label{SimpLin}
We want to solve the eigenvalue problem
\[
L_\ep g = ((u_\ep - v_\ep)g)' + Hg = \lambda(\ep)g.
\]
Let $g = h'$. Then the antiderivative of the above is
\begin{equation}\label{eigen-h}
(u_\ep - v_\ep)h' + Hh = \lambda(\ep)h \pmod 1.
\end{equation}
Let $h = \tilde h \circ \phi_\ep$, where $\phi_\ep$ satisfies
\begin{equation}\label{phi-ep'}
\phi_\ep' = \frac{2\pi}{u_\ep - v_\ep}\left( \int_0^{2\pi} \frac{dy}{u_\ep(y) - v_\ep} \right)^{-1}.
\end{equation}
Then
\[
(u_\ep - v_\ep)\phi_\ep'(\tilde h' \circ \phi_\ep) + H(\tilde h \circ \phi_\ep)
= \lambda(\ep)\tilde h \circ \phi_\ep \pmod 1.
\]
When $\ep$ is small enough, $\phi_\ep$ is a diffeomorphism of $\R/2\pi\Z$, so
\[
2\pi\left( \int_0^{2\pi} \frac{dy}{u_\ep(y) - v_\ep} \right)^{-1}\tilde h'
+ H(\tilde h \circ \phi_\ep) \circ \phi_\ep^{-1} = \lambda(\ep)\tilde h
\pmod 1.
\]
By the change of variable $z = \phi_\ep(y)$,
\begin{align*}
H(\tilde h \circ \phi_\ep) \circ \phi_\ep^{-1}(x)
&= \frac{1}{2\pi}\int_0^{2\pi} \tilde h(\phi_\ep(y))\cot\frac{\phi_\ep^{-1}(x) - y}{2}dy\\
&= \frac{1}{2\pi}\int_0^{2\pi} \tilde h(z)\cot\left( \frac{\phi_\ep^{-1}(x) - \phi_\ep^{-1}(z)}{2} \right)(\phi_\ep^{-1})'(z)dz.
\end{align*}
The convolution kernel of the operator
\[
R_\ep\tilde h = H(\tilde h \circ \phi_\ep) \circ \phi_\ep^{-1} - H\tilde h
\]
is
\begin{equation}\label{R-kernel}
K_\ep(x, z)
= \cot\left( \frac{\phi_\ep^{-1}(x) - \phi_\ep^{-1}(z)}{2} \right)(\phi_\ep^{-1})'(z) - \cot\frac{x - z}{2}
\end{equation}
and the $\ep$-derivative of the kernel is
\begin{align*}
\partial_\ep K_\ep(x, z)
&= -\csc^2\left( \frac{\phi_\ep^{-1}(x) - \phi_\ep^{-1}(z)}{2} \right)
\frac{\partial_\ep\phi_\ep^{-1}(x) - \partial_\ep\phi_\ep^{-1}(z)}{2}
(\phi_\ep^{-1})'(z)\\
& +\cot\left( \frac{\phi_\ep^{-1}(x) - \phi_\ep^{-1}(z)}{2} \right)
\partial_\ep(\phi_\ep^{-1})'(z).
\end{align*}
Near $x = 0$, $\csc x - 1/x^2$ and $\cot x - 1/x$ are smooth, and $(\phi_\ep^{-1})'$ is smooth everywhere, so when $x - z$ is small enough, up to a smooth function in $(x, z)$,
\[
\begin{aligned}
\frac{\partial_\ep K_\ep(x, z)}{2}
&= -\frac{(\partial_\ep\phi_\ep^{-1}(x) - \partial_\ep\phi_\ep^{-1}(z))(\phi_\ep^{-1})'(z)}{(\phi_\ep^{-1}(x) - \phi_\ep^{-1}(z))^2}
+ \frac{\partial_\ep(\phi_\ep^{-1})'(z)}{\phi_\ep^{-1}(x) - \phi_\ep^{-1}(z)}\\
&= \frac{\partial_\ep(\phi_\ep^{-1})'(z)(\phi_\ep^{-1}(x) - \phi_\ep^{-1}(z)) - (\partial_\ep\phi_\ep^{-1}(x) - \partial_\ep\phi_\ep^{-1}(z))(\phi_\ep^{-1})'(z)}{(\phi_\ep^{-1}(x) - \phi_\ep^{-1}(z))^2}\\
&= \frac{\partial_\ep(\phi_\ep^{-1})'(z)(x-z)^2\int_0^1 (1-t)(\phi_\ep^{-1})''((1 - t)z + tx)dt}{(\phi_\ep^{-1}(x) - \phi_\ep^{-1}(z))^2}\\
&- \frac{(\phi_\ep^{-1})'(z)(x-z)^2\int_0^1 (1-t)\partial_\ep(\phi_\ep^{-1})''((1 - t)z + tx)dt}{(\phi_\ep^{-1}(x) - \phi_\ep^{-1}(z))^2}
\end{aligned}
\]
which is itself a smooth function of $(x, z)$ when $x - z$ is small enough
(because $\phi_\ep^{-1}$ is smooth). Then
\[
\|\partial_\ep R_\ep\tilde h^{(m)}\|_{\dot H^k} \lesssim_{k,m} \|\tilde h\|_{L^2/(1)},\quad k, m = 0, 1, \dots,
\]
where the constant does not depend on $\ep$, for all $\tilde h \in H^m/(1)$, or, equivalently,
\begin{equation}\label{R-bound}
\|\partial_\ep R_\ep\tilde h\|_{\dot H^k} \lesssim_{k,m} \|\tilde h\|_{\dot H^{-m}},\quad k, m = 0, 1, \dots,
\end{equation}
where the dot over $H$ means that the norm does not measure frequency zero.
\begin{df}
We say an operator is of class $\mathcal S$ if it satisfies (\ref{R-bound}).
We say a family of operators is of class $\mathcal S$ uniformly if for each $k$ and $m$ there is an implicit constant that makes (\ref{R-bound}) true for all operators in the family.
\end{df}

Thus $\partial_\ep R_\ep$ is of class $\mathcal S$ uniformly in $\ep$. Since $R_0 = 0$, $R_\ep/\ep$ is also of class $\mathcal S$ uniformly in $\ep$.

Now the eigenvalue problem for $\tilde h$ is of the form
\[
(c_\ep\partial_x + H + R_\ep)\tilde h = \lambda(\ep)\tilde h \pmod 1
\]
or, equivalently,
\begin{equation}\label{eigen-tld-h}
(\partial_x + c_\ep^{-1}H + c_\ep^{-1}R_\ep)\tilde h = c_\ep^{-1}\lambda(\ep)\tilde h \pmod 1
\end{equation}
where
\begin{equation}\label{c-ep-def}
c_\ep = 2\pi\left( \int_0^{2\pi} \frac{dy}{u_\ep(y) - v_\ep} \right)^{-1}
\end{equation}
and $R_\ep/\ep$ is of class $\mathcal S$ uniformly in $\ep$. Note that since $u_\ep$ and $v_\ep$ are analytic functions of $\ep$ on a neighborhood of 0,
with $u_0 = 0$ and $v_0 = -1$, so are $\phi_\ep$, $R_\ep$ and $c_\ep$ with $\phi_0 = I$, $R_0 = 0$ and $c_0 = 1$.

\subsection{Spectral analysis of the linearization}\label{SpecLin}
The eigenvalue problem (\ref{eigen-tld-h}) is a perturbation of the eigenvalue problem
\[
\tilde h' + H\tilde h = \lambda\tilde h \pmod 1
\]
with explicit eigenvalues
\[
0\text{ (double)},\ ni,\quad n = \pm1, \pm2, \dots
\]
and eigenfunctions
\[
e^{\pm ix},\ e^{i(n+\sgn n)x},\quad n = \pm1, \pm2, \dots.
\]
They form an othorgonal basis of $H^k/(1)$ for any nonnegative integer $k$.
\begin{df}
Let $T: \dot H^k(\T)\to\dot H^k(\T)$, for $k\in \N$, a linear operator. We will denote
\[
\|T\| := \|T\|_{\dot H^k(\T)\to\dot H^k(\T)}.
\]
\end{df}

The resolvent $(\partial_x + H - z)^{-1}$ is also a Fourier multiplier whose action on Fourier modes is
\begin{equation}\label{L0-res-mult}
(\partial_x + H -z)^{-1}e^{\pm i(n+1)x}
= (\pm ni - z)^{-1}e^{\pm i(n+1)x},\quad n = 0, 1, \dots.
\end{equation}
The circle
\[
\Gamma_n = \{z: |z - ni| = 1/2\}, \quad n = \pm1, \pm2, \dots
\]
encloses a single eigenvalue $\pm ni$, and the circle
\[
\Gamma_0 = \{z: |z| = 1/2\}
\]
encloses the double eigenvalue 0. On $\Gamma_n$ and $\Gamma_0$ we have that
\begin{equation}\label{Gamma-lambda}
|z - mi| \ge 1/2,\quad m \in \Z
\end{equation}
so by (\ref{L0-res-mult}),
\begin{equation}\label{L0-res-bound}
\|(\partial_x + H - z)^{-1}\| \le 2,\quad z \in \Gamma_n,\ n \in \Z.
\end{equation}
Moreover the projection
\[
P_n = -\frac{1}{2\pi i}\int_{\Gamma_n} (\partial_x + H - z)^{-1}dz,\quad
n = \pm1, \pm2, \dots
\]
is the projection on the span of $e^{i(n+\sgn n)x}$ and the projection
\[
P_0 = -\frac{1}{2\pi i}\int_{\Gamma_0} (\partial_x + H - z)^{-1}dz
\]
is the projection on the span of $e^{ix}$ and $e^{-ix}$.

Now when $\ep$ is small enough and $z \in \G_n$, we have that
\[
\partial_x + c_\ep^{-1}H + c_\ep^{-1}R_\ep - z
= (\partial_x + H - z)(1 + (\partial_x + H - z)^{-1}R_\ep')
\]
where
\begin{equation}\label{R-ep'}
R_\ep' = (\partial_x + c_\ep^{-1}H + c_\ep^{-1}R_\ep) - (\partial_x + H)
= (c_\ep^{-1}-1)H + c_\ep^{-1}R_\ep
\end{equation}
is analytic in $\ep$ near 0, with $R_0' = 0$, thanks to the analyticity of $c_\ep$.
Taking the inverse gives that
\[
(\partial_x + c_\ep^{-1}H + c_\ep^{-1}R_\ep - z)^{-1}
= (1 + (\partial_x + H - z)^{-1}R_\ep')^{-1}(\partial_x + H - z)^{-1}
\]
and the Neumann series
\begin{equation}\label{Von-Neumann}
(1 + (\partial_x + H - z)^{-1}R_\ep')^{-1}
= \sum_{n=0}^\infty ((\partial_x + H - z)^{-1}R_\ep')^n
\end{equation}
converges because
\[
\|(\partial_x + H - z)^{-1}R_\ep'\| \le 2\|R_\ep'\| \lesssim_k \ep < 1
\]
when $\ep$ is small enough (depending on $k$). Moveover,
\[
\|(1 + (\partial_x + H - z)^{-1}R_\ep')^{-1} - I\| \lesssim_k \ep
\]
and so
\[
\|(\partial_x + c_\ep^{-1}H + c_\ep^{-1}R_\ep - z)^{-1} - (\partial_x + H - z)^{-1}\| \lesssim \ep
\]
uniformly for $z \in \Gamma_n$. Hence the projections
\begin{equation}\label{Qn-ep}
Q_n(\ep) = -\frac{1}{2\pi i}\int_{\Gamma_n} (\partial_x + c_\ep^{-1}H + c_\ep^{-1}R_\ep - z)^{-1}dz,\quad n \in \Z
\end{equation}
exist and satisfy
\begin{equation}\label{Qn-ep-Qn}
\|Q_n(\ep) - P_n\| \lesssim_k \ep,\quad n \in \Z
\end{equation}
uniformly in $n$. Then by Chapter I, Section 4.6 of \cite{kato},
when $\ep$ is small enough, $Q_n(\ep)$ is conjugate to $P_n$.
Thus $\dim \ran Q_n(\ep) = 1$ for $n \neq 0$ and $\dim \ran Q_0(\ep) = 2$. So
 $\partial_x + c_\ep^{-1}H + c_\ep^{-1}R_\ep$ has a single eigenvalue enclosed by $\Gamma_n$ for $n\neq 0$. In section \ref{lambda=0} we showed that the action on the range of $Q_0(\ep)$ is given by a nonzero nilpotent 2 by 2 matrix.  If $z$ is outside all these circles, then (\ref{L0-res-bound}) still holds and the Neumann series (\ref{Von-Neumann}) still converges to show that $\partial_x + c_\ep^{-1}H + c_\ep^{-1}R_\ep - z$ is invertible, so it has no other eigenvalues.

\subsection{Analyticity of eigenvalues and eigenvectors}
By (\ref{L0-res-mult}) and (\ref{Gamma-lambda}), $(\partial_x + H - z)^{-1}$ is analytic in $(z, \ep)$ for $z$ in a neighborhood $U$ of $\cup_{n\in\Z} \Gamma_n$, and $\ep$ near 0.
By (\ref{R-ep'}), $R_\ep'$ is analytic in $\ep$ near 0, so the series (\ref{Von-Neumann}) shows that $(\partial_x + c_\ep^{-1}H + c_\ep^{-1}R_\ep - z)^{-1}$ is analytic in $(z, \ep)$ for $z \in U$ and $\ep$ near 0, and the integral (\ref{Qn-ep}) shows that all the projections $Q_n(\ep)$ ($n \in \Z$) are analytic in a neighborhood of 0 independent of $n$.

Let $\psi_n(\ep)$ be the corresponding eigenvectors to $Q_n(\ep)$ for $n\neq 0$. Thanks to (\ref{Qn-ep-Qn}), a good choice is $\psi_n(\ep) = Q_n(\ep)e^{i(n+\sgn n)x}$, which is nonzero and analytic in a neighborhood of 0 independent of $n$. Then by (\ref{eigen-tld-h}),
\[
Q_n(\ep)(\partial_x + c_\ep^{-1}(H + R_\ep))e^{i(n+\sgn n)x}
= (\partial_x + c_\ep^{-1}(H + R_\ep))\psi_n(\ep)
= c_\ep^{-1}\lambda_n(\ep)\psi_n(\ep).
\]
On the other hand, the left-hand side equals
\[
(n + \sgn n)iQ_n(\ep)e^{i(n+\sgn n)x} + c_\ep^{-1}Q_n(\ep)(H + R_\ep)e^{\pm i(n+\sgn n)x}
\]
which is another vector analytic in $\ep$ near 0. Then by the next lemma,
all the eigenvalues $c_\ep^{-1}\lambda_n(\ep)$, and hence $\lambda_n(\ep)$,
are analytic in a neighborhood of 0 independent of $n$.

\begin{lem}\label{vec-div-analytic}
Let $u(\ep)$ and $v(\ep)$ be two vectors analytic in $\ep \in U$ satisfying
\[
u(\ep) \neq 0\text{ and }v(\ep) = \lambda(\ep)u(\ep),\quad \ep \in U.
\]
Then $\lambda(\ep)$ is analytic in $\ep \in U$.
\end{lem}
\begin{proof}
Without loss of generality assume that $0 \in U$.
Since the result is local in $\ep$, it suffices to show that $\lambda(\ep)$ is analytic in a smaller neighborhood of 0.

Since $u(0) \neq 0$, we can find a linear functional $f$ such that $f(u(0)) \neq 0$. Then $f(u(\ep)) \neq 0$ in a neighborhood of 0, and so
\[
\lambda(\ep) = \frac{f(v(\ep))}{f(u(\ep))}
\]
is analytic in a neighborhood of 0.
\end{proof}

Regarding the double eigenvalue 0, in section \ref{lambda=0} we showed that $u_\ep'$ and $\partial_\ep u_\ep$ are two generalized eigenvectors of the operator $L_\ep$. Using the relation given in section \ref{SimpLin}, they correspond to two generalized eigenvectors $\psi_0^-(\ep)$ and $\psi_0^+(\ep)$ of the operator $\partial_x + c_\ep^{-1}H + c_\ep^{-1}R_\ep$, via the relation $(\psi_0^-(\ep)\circ\phi_\ep)' = u_\ep'$ and $(\psi_0^+(\ep)\circ\phi_\ep)' = \partial_\ep u_\ep$. Then clearly $\psi_0^\pm(\ep)$ are both analytic in $\ep$.

From the analyticity of the eigenvalues $c_\ep^{-1}\lambda_n(\ep)$,
it is easy to derive bounds on their Taylor coefficients.
\begin{prop}\label{lambda-Taylor-bound}
For $k \ge 1$ and $n \neq 0$, the coefficient of $\ep^k$ in $c_\ep^{-1}\lambda_n(\ep)$ is bounded in absolute value by $C^k$ for a constant $C > 0$ independent of $n$,
\end{prop}
\begin{proof}
At the end of section \ref{SpecLin} we showed that when $\ep$ is in a neighborhood of 0 independent of $n$, the eigenvalues $c_\ep^{-1}\lambda_n(\ep)$ are enclosed in the circle $\Gamma_n$. Then
\[
|c_\ep^{-1}\lambda_n(\ep) - ni| < 1/2,\quad n = \pm1, \pm2, \dots.
\]
The result follows from Cauchy's integral formula for Taylor coefficients.

\end{proof}

\begin{cor}\label{lambda-Taylor-bound-cor}
For $k \ge 0$ and $n \neq 0$, the coefficient of $\ep^k$ in $\lambda_n(\ep)$ is bounded in absolute value by $|n|C^k$ for a constant $C > 0$ independent of $n$,
\end{cor}
\begin{proof}
Since $c_\ep$ is analytic in $\ep$ near 0 with $c_0 = 1$, and $\lambda_n(0) = ni$, the result follows from Leibniz's rule.
\end{proof}

\subsection{Conjugation to a Fourier multiplier}
We have conjugated the eigenspaces of $T = \partial_x + c_\ep^{-1}H + c_\ep^{-1}R_\ep$ (and also of $c_\ep\partial_x + H + R_\ep$) to Fourier modes via the operator
\[
1 + W_\ep = \sum_{n\in\Z} P_nQ_n(\ep)
\]
where $P_0$ is the projection onto the span of $e^{\pm ix}$, $Q_0(\ep)$ is the projection onto the span of $\psi_0^\pm(\ep)$, $P_n$ is the projection onto the span of $e^{i(n + \sgn n)x}$, $Q_n(\ep)$ is the projection onto the span of $\psi_n(\ep)$, $n = \pm1, \pm2, \dots$.

We will view $T$ as a perturbation of $\partial_x + c_\ep^{-1}H$ and follow the proof of Chapter V, Theorem 4.15a in \cite{kato}. In the process we will extract more information from the fact that $R_\ep$ is of class $\mathcal S$. Since
\begin{equation}\label{sum-Pn=1}
P_n^2=P_n,\quad \sum_{n\in\Z} P_n = 1,
\end{equation}
we have that
\begin{equation}\label{Wep-def}
W_\ep = \sum_{n\in\Z} P_n(Q_n(\ep) - P_n)
\end{equation}
and $W_0 = 0$.

\begin{prop}
$W_\ep/\ep$ is of class $\mathcal S$ uniformly in $\ep$.
\end{prop}
\begin{proof}
We bound each term on the right-hand side separately. By Chapter V, (4.38) of \cite{kato},
\[
Q_n(\ep) - P_n = -c_\ep^{-1}Q_n(\ep)R_\ep Z_n(\ep) - c_\ep^{-1}Z_n'(\ep)R_\ep P_n
\]
where
\[
Z_n(\ep) = \frac{1}{2\pi i}\int_{\Gamma_n} (z - (n + (1 - c_\ep^{-1})\sgn n)i)^{-1}(\partial_x + c_\ep^{-1}H - z)^{-1}dz
\]
and
\[
Z_n'(\ep) = \frac{1}{2\pi i}\int_{\Gamma_n} (z - c_\ep^{-1}\lambda_n(\ep))^{-1}(T - z)^{-1}dz.
\]
We now bound the operator norms of the right-hand side, with uniformity in $\ep$ and decay in $n$, in order to show that the sum in $n$ converges.

First note that it is clear from the frequency side that when $\ep$ is in a neighborhood of 0 independent of $n$ and $z \in \cup_{n\in\Z} \Gamma_n$,
for all $m \ge 0$, the operator $(\partial_x + c_\ep^{-1}H - z)^{-1}$ is bounded from $H^m$ to $H^m$, uniformly in $\ep$ and $z$.
Since $R_\ep/\ep$ is of class $\mathcal S$ uniformly in $\ep$ (see \eqref{R-bound} and notice that $R_0=0$),
it follows from the Neumann series that $\|(T - z)^{-1}\|_{\dot H^m\to\dot H^m}$ is finite and only depends on $m$. Since $|z - (n + (1 - c_\ep^{-1})\sgn n)i|$ and $|z - c_\ep^{-1}\lambda_n(\ep)|$ are uniformly bounded from below, both $Z_n(\ep)$ and $Z_n'(\ep)$ are bounded from $\dot H^m$ to $\dot H^m$, uniformly in $\ep$ and $n$.
Since $Q_n(\ep)$ is given by a similar integral (\ref{Qn-ep}),
it also has this property, which is also trivially true for $P_n$.
Now for all $n, m, k \in \Z$, $m, k \ge 0$ and $\tilde h \in L^2$,
\begin{equation}\label{Qn-Pn-bound1}
\begin{aligned}
\|Z_n'(\ep)R_\ep P_n\tilde h\|_{\dot H^k}
&\lesssim_k \|R_\ep P_n\tilde h\|_{\dot H^k}
\lesssim_{m,k} |\ep|\|P_n\tilde h\|_{\dot H^{-m-2}}\\
&\lesssim_{m,k} |\ep|(1 + |n|)^{-2}\|\tilde h\|_{\dot H^{-m}}
\end{aligned}
\end{equation}
because $P_n$ is the projection onto very specific Fourier modes.
For the first term we have
\[
\|R_\ep Z_n(\ep)\tilde h\|_{\dot H^k}
\lesssim_{m,k} |\ep|\|Z_n(\ep)\tilde h\|_{\dot H^{-m}}
\lesssim_{m,k} |\ep|\|\tilde h\|_{\dot H^{-m}}.
\]
To introduce the action of $Q_n(\ep)$, note that the image of $Q_n(\ep)$ lies in the eigenspace of the operator $c_\ep\partial_x + H + R_\ep$,
with eigenvalue $\lambda_n(\ep)$, so for $n \neq 0$ and $u \in \Im Q_n(\ep)$ we have
\[
u = \lambda_n(\ep)^{-1}(c_\ep u' + Hu + R_\ep u)
\]
so $\|u\|_{\dot H^k} \lesssim_k |\lambda_n(\ep)|^{-1}\|u\|_{\dot H^{k+1}}
\lesssim |n|^{-1}\|u\|_{\dot H^{k+1}}$. Hence
\begin{equation}\label{Qn-Pn-bound2}
\|Q_n(\ep)R_\ep Z_n(\ep)\|_{\dot H^k}
\lesssim_k n^{-2}\|R_\ep Z_n(\ep)\|_{\dot H^{k+2}}
\lesssim_{m,k} |\ep|(1 + |n|)^{-2}\|\tilde h\|_{\dot H^{-m}}.
\end{equation}
This also holds for $n=0$ because $R_\ep/\ep$ is of class $\mathcal S$ uniformly, so $W_\ep/\ep$ is of class $\mathcal S$ uniformly in $\ep$ thanks to the convergence of $\sum_{n\in\Z} (1 + |n|)^{-2}$.
\end{proof}

Now for $k = 0, 1, \dots$, there is a neighborhood of 0 such that when $\ep$ is in this neighborhood, $\|W_\ep\|_{\dot H^k\to\dot H^k} < 1$, so $1 + W_\ep: \dot H^k \to \dot H^k$ is invertible. By (\ref{sum-Pn=1}) and (\ref{Wep-def}) it follows easily that
\begin{equation}\label{conjugate}
(1 + W_\ep)Q_n(\ep) = P_n(1 + W_\ep)
\end{equation}
so the eigenspaces of $T$ is conjugated to the (span of) Fourier modes,
and hence $T$ is conjugated to a Fourier multiplier.

We have proven the following lemma:

\begin{lem}\label{conjugado} For $\ep$ small enough, there exists an operator $W_\ep$, such that $W_\ep/\ep$ is of class $\mathcal{S}$, uniformly in $\ep$. Moreover,
\begin{enumerate}
\item $1+W_\ep\,:\, \dot{H}^k\to \dot{H}^k$ is invertible,
\item $(1+W_\ep)Q_n(\ep)=P_n(1+W_\ep)$, $n\in\Z$,
\item If $\psi$ is in the closed linear span of the eigenvectors $\psi_n(\ep)$ ($n \neq 0$) of $c_\ep\partial_x + H + R_\ep$, then
\[
(1+W_\ep)(c_\ep\partial_x+H+R_\ep)\psi=\Lambda_\ep(1+W_\ep)\psi
\]
where $\Lambda_\ep$ is a multiplier such that
\[
\Lambda_\ep e^{i(n+\sgn n)x}=\lambda_n(\ep)e^{i(n+\sgn n)x},\quad
n = \pm1, \pm2, \dots
\]
\end{enumerate}
\end{lem}

\subsection{Taylor expansion of eigenvalues}
Now we Taylor expand the eigenvalues $\lambda_n(\ep)$ for $n \neq 0$.
To do so it is more convenient to study the eigenvalue problem (\ref{eigen-h}) for $h$:
\[
L_\ep g : = \left((u_\ep-v_\ep)g\right)'+Hg=\lambda(\ep)g.
\]
Recall the operator $L = L_0 = \partial_x + H$ whose action on the Fourier modes is
\[
\mathcal F(Lg)(m) = i(m - \sgn m)\hat g(m)
\]
with eigenvalues 0 (double), $\pm i$, $\pm2i$, $\dots$ ($g$ mean zero).

Since $(u_\ep, v_\ep)$ is analytic in $\ep$ on a neighborhood of 0, and
\[
\|h'\|_{L^2} \le \|h' + Hh\|_{L^2} + \|Hh\|_{L^2} = \|Lh\|_{L^2} + \|h\|_{L^2},
\]
by Chapter VII, Theorem 2.6 in \cite{kato}, $L_\ep$ is a holomorphic family of operators of type (A), so by Chapter VII, Section 2.3, all the results in Chapter II, Sections 1 and 2 apply, and we can compute the Taylor coefficients of $\lambda(\ep)$ as if $L_\ep$ acted on a finite dimensional vector space.

We start with computing the resolvent of $L$:
\[
R(z) = (L - z)^{-1}
\]
whose action on the Fourier modes is
\[
\mathcal F(R(z)g)(m) = (i(m - \sgn m) - z)^{-1}\hat g(m).
\]
Around the eigenvalue $ni$ ($n = \pm1, \pm2, \dots$) we have the expansion
\[
R(z) = (ni - z)^{-1}P_n + \sum_{k=0}^\infty (z - ni)^kS_n^{k+1}
\]
where $P_n$ is the projection on the span of $e^{i(n+\sgn n)x}$ and
\begin{equation}\label{Sn-def}
\mathcal F(S_ng)(m) = \frac{\hat g(m)}{i(m - \sgn m - n)},\quad
m \neq n + \sgn n.
\end{equation}

By \cite{kato} (II.2.33),
\[
\lambda_n(\ep) = ni + \sum_{k=1}^\infty \ep^k\lambda_n^{(k)},\quad
n = \pm1, \pm2, \dots
\]
where
\[
\lambda_n^{(k)} = \sum_{p=1}^k \frac{(-1)^p}{p} \sum_{{v_1+\cdots+v_p=n\atop v_j\ge1}\atop h_1+\cdots+h_p=p-1} \Tr L^{(v_p)}S_n^{(h_p)} \cdots L^{(v_1)}S_n^{(h_1)}
\]
where $S_n^{(0)} = -P_n$ and for $h \ge 1$, $S_n^{(h)} = S_n^h$, with $S_n$ defined in (\ref{Sn-def}), and $L^{(v)}$ is the coefficient of $\ep^v$ in the Taylor expansion of $L_\ep$. Note that the constraints in the summation imply that there is some $j \in \{1, \dots, p\}$ such that $h_j = 0$ and so $S_n^{(h_j)} = -P_n$, so every summand is a finite-rank operator whose trace is thus well defined.

\begin{lem}
If $A$ is a finite-rank operator, then $\Tr AB = \Tr BA$.
\end{lem}
\begin{proof}
By linearity we can assume $A$ has the form $A(\cdot) = f(\cdot)v$ for some (not necessarily continuous) linear functional $f$.
Then $\Tr A = f(v)$. Since $AB(\cdot) = f(B\cdot)v$ and $BA(\cdot) = f(\cdot)Bv$, it follows that $\Tr AB = f(Bv) = \Tr BA$.
\end{proof}

Using the lemma above, we can simplify the sum in $\lambda_n^{(k)}$ a little. Indeed, there are $p$ circular rotations of the tuple $(h_1, \dots, h_p)$. Since $(\sum_j h_j, p) = 1$, the $p$ circular rotations are all distinct, so we can choose the lexicographically smallest one as a representative. For such a representative $h_1 = \min_j h_j = 0$, so $S_n^{(h_1)} = -P_n$, so we only need to act $L^{(v_p)}S_n^{(h_p)} \cdots L^{(v_1)}$ on $e^{i(n+\sgn n)x}$ and take the $(n + \sgn n)$-th mode to compute the trace. Thus
\begin{equation}\label{lambda-n}
\lambda_n^{(k)} = \sum_{p=1}^k (-1)^{p-1} \sum_{{v_1+\cdots+v_p=k\atop v_j\ge1}\atop{h_1+\cdots+h_p=p-1\atop{(h_1,\dots,h_p)\atop\text{is a representative}}}} \mathcal F[L^{(v_p)}S_n^{(h_p)} \cdots L^{(v_1)}e^{i(n+\sgn n)x}](n + \sgn n).
\end{equation}

Let us compute some terms  $\lambda_n^{(k)}$ by using the formula \eqref{lambda-n}. We have that

\[
\lambda_n^{(1)} = \Tr L^{(1)}P_n = 0
\]
because $L_1$ shifts the mode by 1, and
\[
\lambda_n^{(2)} = \Tr(L^{(2)}P_n - L^{(1)}S_nL^{(1)}P_n)
\]

Put $s = \sgn n$. We extract the $(n+s)$-th mode of each term:
\begin{align*}
\Tr L^{(2)}P_n &= \mathcal F[L^{(2)}2e^{i(n+s)x}](n+s) = \frac{i(n+s)}{4}\\
L^{(1)}S_nL^{(1)}e^{i(n+s)x}
&= \frac{iL^{(1)}S_n}{2}((n+s+1)e^{i(n+s+1)x} + (n+s-1)e^{i(n+s-1)x})\\
&= \frac{L^{(1)}}{2}((n+s+1)e^{i(n+s+1)x} - (n+s-1)e^{i(n+s-1)x})\\
\Tr L^{(1)}S_nL^{(1)}P_n &= \frac{i(n+s+1)(n+s)-i(n+s-1)(n+s)}{4}
= \frac{i(n+s)}{2}
\end{align*}

so
\[
\lambda_n^{(2)} = \frac{i(n+s)}{4} - \frac{2i(n+s)}{4} = -\frac{i(n+s)}{4}.
\]
We can further compute that
\[
\lambda_n(\ep) = in - \frac{\ep^2i(n+s)}{4} - \frac{11\ep^4i(n+s)}{32}
- \frac{527i\ep^6(n+s)}{768} + O_n(\ep^7)
\]
for $n = \pm 1, \pm2, \pm3, \dots$.

\begin{prop}\label{lambda-Taylor}
For $n = \pm1, \pm2, \dots$,
\[
\lambda_n^{(k)} =
\begin{cases}
0, & 2 \nmid k\\
ic^{(k)}(n + \sgn n), & k \le 2|n| + 2
\end{cases}
\]
where $c^{(k)}$ is the $k$-th Taylor coefficient of $c_\ep$ as defined in (\ref{c-ep-def}).

When $k \ge 2|n| + 4$, $\lambda_n^{(k)}$ is still purely imaginary but the formula $\lambda_n^{(k)}=ic^{(k)}(n+\sgn n)$ does not hold in general.
\end{prop}
\begin{proof}
Firstly we notice that, for $n = \pm1$ the coefficient of $\ep^6$ in $\lambda_{\pm 1}(\ep)$ is,
\[
\lambda_1(\ep) = i - \frac{\ep^2i}{2} - \frac{11\ep^4i}{16} - \frac{529\ep^6i}{384} + O(\ep^7),
\]
which does not hold $\lambda_{\pm 1}^{(6)}=\pm 2 ic^{(6)}.$

Next, we prove the fist part of the lemma. In each summand of (\ref{lambda-n}), all the coefficients are real,
except that each operator $L$ brings a factor of $i$ to the Fourier coefficients (via the operator $\partial_x$),
and each operator $S_n$ removes a factor of $i$ (see (\ref{Sn-def})).
Hence each summand is purely imaginary, and so is $\lambda_n^{(k)}$.

In each summand of (\ref{lambda-n}), the operator $S_n^{(h_j)}$ is a Fourier multiplier that does not shift the modes, while the operator $L^{(m)}g = \left((u^{(m)} - v^{(m)})g\right)'$ shifts the modes by at most $m$ because $u^{(m)}$ only contains modes up to $e^{\pm imx}$. Also the amount of shift $= m \pmod 2$. Thus when acting the sequence $L^{(v_p)}S_n^{(h_p)} \cdots L^{(v_1)}$ on $e^{i(n+s)x}$, the mode is consecutively shifted by at most $v_1, v_2, \dots, v_p$, and the total amount of shifts $= \sum_j v_j = k \pmod 2$. Since in the end we are taking the $(n + s)$-th mode, the total amount of shifts must be 0 in order to count, so when $k$ is odd $\lambda_n^{(k)} = 0$. When $k$ is even, the mode $e^{i(n+s)x}$ can only be shifted as far as $e^{i(n+s\pm k/2)x}$; otherwise it can never be shifted back. Hence when $k \le 2|n| + 2 = 2|n + s|$, the frequency always has the same sign as $n$ or becomes 0. In the former case we can take $\sgn m = \sgn n$ in (\ref{Sn-def}), while in the latter case the derivative in $L$ kills it, so it does not hurt if we still take $\sgn m = \sgn n$ in (\ref{Sn-def}). Either way we can take $\sgn m = \sgn n$ in (\ref{Sn-def}). Thus the action of $S_n$ is the same as that of $S_n'$, where
\[
\mathcal F(S_n'g)(m) = \frac{\hat g(m)}{i(m - n - \sgn n)},\quad
m \neq n + \sgn n.
\]

For $n > 0$, the operator $S_n'$ is the analog of $S_n$ for $L^+$ with
\[
\mathcal F(L^+g)(m) = i(m - 1)\hat g(m),
\]
i.e., $L^+g = g' - ig$. Hence $\lambda_n^{(k)}$ remains the same if we replace $L$ with $L^+$. Now we have that
\[
L_\ep^+g := L^+g + \sum_{n=1}^\infty \ep^nL^{(n)}g
= -v_\ep g' - ig + (u_\ep g)' = ((u_\ep - v_\ep)g)' - ig
\]
whose eigenvalue problem is
\[
((u_\ep - v_\ep)g)' - ig = \lambda^+(\ep)g.
\]
Using the same change of variable as in Section \ref{SimpLin},
the problem above can be transformed to
\[
\tilde h' - ic_\ep^{-1}\tilde h = c_\ep^{-1}\lambda^+(\ep)\tilde h
\]
whose eigenvalues are
\[
\lambda_{n'}^+(\ep) = n'c_\ep i - i.
\]
Since when $\ep \to 0$, $\lambda_n(\ep) \to ni$ and $c_\ep \to 1$,
we must have that $n' = n + 1$, and so
\[
\lambda_n(\ep) = (n + 1)c_\ep i - i + O_n(\ep^{2n+4}).
\]

For $n < 0$, note that since $L$ preserves real-valued functions,
its eigenvalues come in conjugate pairs, so $\lambda_n(\ep) = \overline{\lambda_{|n|}(\ep)} = -\lambda_{|n|}(\ep)$ has the same property.
\end{proof}

\begin{cor}\label{lambda-remainder}
When $\ep$ is small enough,
\begin{align*}
|\lambda_n(\ep) - (n + \sgn n)c_\ep i + i\sgn n| &< |n|(C\ep)^{2|n|+4}
< C'\ep^6, \quad n \in \Z \backslash \{0\},\\
|\lambda_n'(\ep) - (n + \sgn n)\partial_\ep c_\ep i| &< |n|(C\ep)^{2|n|+3}
< C'\ep^5, \quad n \in \Z \backslash \{0\}
\end{align*}
for some constant $C$, $C' > 0$ independent of $n$.
\end{cor}
\begin{proof}
By Proposition \ref{lambda-Taylor}. the Taylor expansions of $\lambda_n(\ep)$ and $(n + \sgn n)c_\ep i - i\sgn n$ differ only from the term $\ep^{2|n|+4}$. By Corollary \ref{lambda-Taylor-bound-cor}, the error terms of the former sum up to $O(|n|\sum_{k=2|n|+4}^\infty (C\ep)^k) = O(|n|(C\ep)^{2|n|+4})$ if, say, $C|\ep| < 1/2$. The error term of the latter clearly also satisfy this bound.

To extend the chain of inequalities it suffices to note that $|n|(C\ep)^{2|n|-2}$ is uniformly bounded for $n \neq 0$ if $|C\ep| < 1/2$.
\end{proof}

\subsection{Time resonance analysis}
For $m$, $n$ and $l \in \Z$ we consider
\[
\lambda_m(\ep) + \lambda_n(\ep) + \lambda_l(\ep)
= (m + n + l)c_\ep i + (\sgn m + \sgn n + \sgn l)(c_\ep - 1)i + O(\ep^6).
\]

\begin{prop}\label{non-res}
If $m$, $n$, $l \in \Z$ and $mnl \neq 0$, then when $\ep$ is small enough,
$|\lambda_m(\ep) + \lambda_n(\ep) + \lambda_l(\ep)| > \ep^2/5$.
\end{prop}
\begin{proof}
By (\ref{c-ep-def}) and (\ref{uv-Taylor-exp}),
\begin{align*}
\frac{c_\ep}{2\pi} &= \left( \int_0^{2\pi} \frac{dy}{u_\ep(y) - v_\ep} \right)^{-1}
= \left( \int_0^{2\pi} \frac{dy}{1 + \ep\cos y - \frac{\ep^2}{2}\cos2y + \frac{\ep^2}{4}} \right)^{-1} + O(\ep^3)\\
&= \left( \int_0^{2\pi} \left( 1 - \ep\cos y + \ep^2\cos^2y + \frac{\ep^2}{2}\cos2y - \frac{\ep^2}{4} \right)dy \right)^{-1} + O(\ep^3),\\
c_\ep &= (1 + \ep^2/4)^{-1} + O(\ep^3) = 1 - \ep^2/4 + O(\ep^3).
\end{align*}

We distinguish three cases.

{\bf Case 1:} $m + n + l \neq 0$. Then $|m + n + l| \ge 1$. Since $c_\ep - 1 \lesssim \ep^2$,
\[
\lambda_m(\ep) + \lambda_n(\ep) + \lambda_l(\ep)
= (m + n + l)c_\ep i + O(\ep^2).
\]
Since $c_\ep \to 1$ as $\ep \to 0$, $|\lambda_m(\ep) + \lambda_n(\ep) + \lambda_l(\ep)| > |m + n + l|/2$ for small $\ep$.

{\bf Case 2:} $m + n + l = 0$ and $mnl \neq 0$. Then
\[
\lambda_m(\ep) + \lambda_n(\ep) + \lambda_l(\ep)
= -\frac{1}{4}(\sgn m + \sgn n + \sgn l)\ep^2i + O(\ep^3).
\]
Since $|\sgn m| = |\sgn n| = |\sgn l| = 1$, we have that $|\sgn m + \sgn n + \sgn l| \ge 1$, so $|\lambda_m(\ep) + \lambda_n(\ep) + \lambda_l(\ep)| > \ep^2/5$ when $\ep$ is small enough.
\end{proof}

When $m + n + l = 0$ and $mnl = 0$, since $\lambda_n(\ep)$ is odd in $n$,
it follows that $\lambda_m(\ep) + \lambda_n(\ep) + \lambda_l(\ep) = 0$.
We do have time resonance in this case.
We will eliminate this case by choosing a new frame of reference.

\section{A new frame of reference}\label{NFR}
Recall that the traveling wave solution
\[
f_\ep(x, t) = u_\ep(x + v_\ep t)
\]
satisfies
\[
\partial_tf_\ep = Hf_\ep + f_\ep\partial_xf_\ep,
\]
i.e.,
\[
v_\ep u_\ep' = Hu_\ep + u_\ep u_\ep'.
\]

Now we aim to find a new reference frame. Let $P_0^\pm(\ep)$ be the projection on the 1 dimensional space spanned by the eigenvector $\varphi_0^+(\ep) = \partial_\ep u_\ep$ and $\varphi_0^-(\ep) = -\ep^{-1}u_\ep'$, respectively. Then we aim to rewrite
\[
f(x, t) = u_{\ep(t)}(x + a(t)) + g(x + a(t), t)
\]
where $\ep$, $a \in \R$ and $P_0^\pm(\ep(t))g = 0$. We first show that it is always possible, provided that $f$ is close to a traveling wave.
\begin{prop}\label{new-frame}
Let $k \ge 2$. Then there is $r = r(k) > 0$ such that if $|\ep_0| < r$ and $\|f - u_{\ep_0}\|_{H^k} < r|\ep_0|$, then there is $\ep \in \R$, $a \in \R/2\pi\Z$ and $g \in H^k$ such that
\begin{align}
\label{f=u+g}
f(x) &= u_\ep(x + a) + g(x + a),\\
\label{P0g=0}
P_0^\pm(\ep)g &= 0,\\
\label{g-L2-bound}
|\ep - \ep_0| + \|g\|_{H^k} &\lesssim \|f - u_{\ep_0}\|_{H^2}.
\end{align}
Moreover, $\ep$, $a$ and $g$ depend smoothly on $f$.
\end{prop}
\begin{proof}
Define the map $F: (-r, r)^2 \to \R^2$, $(\ep, a) \mapsto (y^+, y^-)$ with
\begin{equation}\label{F-def}
P_0^\pm(\ep)(f(x - a) - u_\ep(x)) = y^\pm\varphi_0^\pm(\ep).
\end{equation}
We now find the solution to the equation $F(\ep, a) = 0$.
Since $P_0^\pm(\ep)$ is uniformly bounded in $L^2$ and $\|\varphi_0^\pm(\ep)\|$ is uniformly bounded from below,
\begin{equation}\label{F-bound}
|F(\ep, a)| \lesssim \|f(x - a) - u_\ep\|_{L^2}.
\end{equation}

Summing the two equations in (\ref{F-def}) and taking the total derivative yields
\begin{align}
\nonumber
&-(P_0^+(\ep) + P_0^-(\ep))(f'(x - a))da - \varphi_0^+(\ep)d\ep\\
\label{dF-LHS}
&+ (\partial_\ep P_0^+(\ep) + \partial_\ep P_0^-(\ep))(f(x - a) - u_\ep(x))d\ep\\
&= \varphi_0^+(\ep)dy^+ + \varphi_0^-(\ep)dy^- + y^+\partial_\ep\varphi_0^+(\ep)d\ep + y^-\partial_\ep\varphi_0^-(\ep)d\ep.
\label{dF-RHS}
\end{align}
Since $\|f\|_{H^2} \le \|u_{\ep_0}\|_{H^2} + r|\ep_0| \lesssim |\ep_0|$, we have that
\begin{align}
\nonumber
\|f(x - a) - u_\ep\|_{H^1}
&\le \|f(x - a) - f(x)\|_{H^1} + \|f - u_{\ep_0}\|_{H^1} + \|u_\ep - u_{\ep_0}\|_{H^1}\\
&\lesssim |a\ep_0| + r|\ep_0| + |\ep - \ep_0|.
\label{f-u-H1}
\end{align}
Since both $P_0^\pm(\ep)$ and $\partial_\ep P_0^\pm(\ep)$ are uniformly bounded on $L^2$, and $u_\ep' = -\ep\varphi_0^-(\ep)$,
\[
\|(\ref{dF-LHS}) - \ep\varphi_0^-(\ep)da + \varphi_0^+(\ep)d\ep\|_{L^2} \lesssim (|a\ep_0| + r|\ep_0| + |\ep - \ep_0|)(|da| + |d\ep|).
\]
By (\ref{F-bound}) and (\ref{f-u-H1}),
\[
\|y^\pm\partial_\ep\varphi_0^\pm(\ep)\|_{L^2} \lesssim |F(\ep, a)|
\lesssim |a\ep_0| + r|\ep_0| + |\ep - \ep_0|
\]
so
\[
\|(\ref{dF-RHS}) - \varphi_0^+(\ep)dy^+ - \varphi_0^-(\ep)dy^-\|_{L^2}
\lesssim (|a\ep_0| + r|\ep_0| + |\ep - \ep_0|)|d\ep|.
\]
Hence the equality between (\ref{dF-LHS}) and (\ref{dF-RHS}) gives an estimate of the differential
\[
\left\| dF(\ep, a) -
\begin{pmatrix}
1 & 0\\
0 & -\ep
\end{pmatrix} \right\| \lesssim |a\ep_0| + r|\ep_0| + |\ep - \ep_0|.
\]

We assume that the solution $(\ep, a)$ satisfies $|\ep - \ep_0| + |a\ep_0| < r_0|\ep_0|$, where $r_0$ is small enough. This in particular implies that $|\ep_0|/2 < |\ep| < 2|\ep_0|$. Then
\[
\left\| dF(\ep,a) -
\begin{pmatrix}
1 & 0\\
0 & -\ep_0
\end{pmatrix} \right\| \lesssim (r_0 + r)|\ep_0|
\]
is also small enough. Let
\[
G = \mathbb{I} + \left( dF(\ep, a) -
\begin{pmatrix}
1 & 0\\
0 & -\ep_0
\end{pmatrix} \right)
\begin{pmatrix}
1 & 0\\
0 & -1/\ep_0
\end{pmatrix}.
\]
Then
\[
dF = G
\begin{pmatrix}
1 & 0\\
0 & -\ep_0
\end{pmatrix}
\]
and
\[
\|G - \mathbb{I}\| \lesssim r_0 + r.
\]
If $r_0$ and $r$ are small enough, then $\|G\|$ and $\|G^{-1}\| < 2$.

Let $(\ep_1, a_1) = (\ep_0, 0) - dF(\ep_0, 0)^{-1}F(\ep_0, 0)$. Then (recalling (\ref{F-bound}))
\[
|\ep_1 - \ep_0| + |a_1\ep_0| \lesssim |G^{-1}F(\ep_0, 0)| \lesssim |F(\ep_0, 0)| \lesssim \|f - u_{\ep_0}\|_{L^2} \lesssim r|\ep_0|.
\]
Since $|\partial_\ep^2F|$ and $|\partial_{a\ep}F| \lesssim 1$, and $|\partial_a^2F| \lesssim \|f\|_{H^2} \lesssim |\ep_0|$, by Taylor's theorem,
\[
|F(\ep_1, a_1)| \lesssim |\ep_1 - \ep_0|^2 + |\ep_1 - \ep_0||a_1| + |\ep_0||a_1|^2 \lesssim r|F(\ep_0, 0)|.
\]
Hence the iteration $(\ep_{n+1}, a_{n+1}) = (\ep_n, a_n) + dF(\ep_n, a_n)^{-1}F(\ep_n, a_n)$ converges when $r$ is small enough. Moreover $|\ep_n - \ep_0| + |a_n\ep_0| \lesssim |F(\ep_0, 0)|$.
Then $(\ep, a) := \lim_{n\to\infty} (\ep_n, a_n)$ satisfies $F(\ep, a) = 0$ and $|\ep - \ep_0| + |a\ep_0| \lesssim |F(\ep_0, 0)| \lesssim r|\ep_0| < r_0|\ep_0|$ if $r$ is small compared to $r_0$.

Let $g = f(x - a) - u_\ep$. Then (\ref{f=u+g}) and (\ref{P0g=0}) clearly hold. Moreover,
\begin{align*}
\|g\|_{H^k} &= \|g(x + a)\|_{H^k} = \|f(x) - u_\ep(x + a)\|_{H^k}\\
&\le \|f - u_{\ep_0}\|_{H^k} + \|u_\ep(x + a) - u_{\ep_0}(x)\|_{H^k}\\
&\lesssim \|f - u_{\ep_0}\|_{H^k} + |\ep - \ep_0| + |a\ep_0|\\
&\lesssim \|f - u_{\ep_0}\|_{H^k} + |F(\ep_0, 0)| \lesssim \|f - u_{\ep_0}\|_{H^k}
\end{align*}
showing (\ref{g-L2-bound}). The smooth dependence of $\ep$, $a$ and $g$ on $f$ is also clear.
\end{proof}

By translation symmetry, if $f$ is $r|\ep_0|$-close to $u_{\ep_0}(x + a)$ for some $a \in \R/2\pi\Z$, we can reach a similar conclusion. Then we can write
\[
f(x, t) = u_{\ep(t)}(x + a(t)) + g(x + a(t), t).
\]
We will obtain an energy estimate for $g$.
Combined with local wellposedness of the equation (\ref{BH}) and Proposition \ref{new-frame}, we can show that the solution extends as long as the energy estimate closes, see the end of section \ref{lifespan1}.

To get the energy estimate, we first need to derive an evolution equation for $g$. Since $f$ is differentiable in $t$, so are $\ep(t)$, $a(t)$ and $g$, and we get
\[
f_t(x, t) = a'(t)(u_\ep' + g_x)(x + a(t)) + \ep'(t)\partial_\ep u_\ep(x + a(t)) + g_t(x + a(t), t)
\]
and
\begin{align*}
(Hf + ff_x)(x, t) &= (Hu_\ep + u_\ep u_\ep')(x + a(t)) + Hg(x + a(t), t)\\
&+\partial_x(u_\ep(x + a(t))g(x + a(t), t)) + (gg_x)(x + a(t), t).
\end{align*}
The equation for $g$ is then
\begin{align*}
g_t &= v_\ep u_\ep' - a'(t)(u_\ep' + g_x) - \ep'(t)\partial_\ep u_\ep + Hg + (u_\ep g)_x + gg_x\\
&= L_\ep g + (v_\ep - a'(t))(u_\ep' + g_x) - \ep'(t)\partial_\ep u_\ep + gg_x.
\end{align*}
Since $P_0^\pm(\ep)g(t) = 0$, we have that $P_0^\pm(\ep)g_t = -\ep'(t)\partial_\ep P_0^\pm(\ep)g$, so the action of the projections $P_0^\pm(\ep)$ on the above equation is
\begin{align*}
(v_\ep - a'(t))P_0^+(\ep)g_x + \ep'(t)(\partial_\ep P_0^+(\ep)g - \partial_\ep u_\ep) + P_0^+(\ep)(gg_x) &= 0,\\
(v_\ep - a'(t))(u_\ep' + P_0^-(\ep)g_x) + \ep'(t)\partial_\ep P_0^-(\ep)g + P_0^-(\ep)(gg_x) &= 0.
\end{align*}
Since $P_0^\pm(\ep)$ are bounded on $L^2$, we have that $\|P_0^\pm(\ep)g_x\|_{L^2} \lesssim \|g\|_{H^1}$. Since $P_0^\pm(\ep)$ are analytic in $\ep$, we have that $\|\partial_\ep P_0^\pm(\ep)g\|_{L^2} \lesssim \|g\|_{L^2}$.
Since $P_0^\pm(\ep)$ is a projection, we have that $P_0^\pm(\ep)^2 = P_0^\pm(\ep)$. Taking the derivative in $\ep$ and using the constraint $P_0^\pm(\ep)g = 0$ we have that $P_0^\pm(\ep)\partial_\ep P_0^\pm(\ep)g = \partial_\ep P_0^\pm(\ep)g$, i.e., $\partial_\ep P_0^\pm(\ep)g$ is in the 1-dimensional space spanned by $\varphi_0^\pm(\ep)$. Hence
\begin{align*}
|P_0^\pm(\ep)g_x/\varphi_0^+(\ep)| &\lesssim \|g\|_{H^1}, &
|\partial_\ep P_0^\pm(\ep)g/\varphi_0^\pm(\ep)| &\lesssim \|g\|_{L^2}.
\end{align*}
Thus, dividing the two equations by $\varphi_0^\pm(\ep)$ we get
\[
\left| \left(
\begin{pmatrix}
0 & 1\\
\ep & 0
\end{pmatrix}
+ O(\|g\|_{H^1})
\right)
\begin{pmatrix}
v_\ep - a'(t)\\
\ep'(t)
\end{pmatrix}
\right| = \left|
\begin{pmatrix}
P_0^+(\ep)(gg_x)/\varphi_0^+(\ep)\\
P_0^-(\ep)(gg_x)/\varphi_0^-(\ep)
\end{pmatrix}
\right| \lesssim \|g(t)\|_{H^1}^2.
\]
Assuming $\|g(t)\|_{H^1}/|\ep|$ is small enough we have that
\begin{equation}\label{ep'-a'-bound}
\begin{pmatrix}
v_\ep - a'(t)\\
\ep'(t)
\end{pmatrix}
=
\begin{pmatrix}
O(\|g(t)\|_{H^1}^2/|\ep|)\\
O(\|g(t)\|_{H^1}^2)
\end{pmatrix}
\end{equation}

\subsection{Diagonalization}
To find the evolution of other modes, we diagonalize the equation for $g$.
Let $g = h_x$ and $h = \tilde h \circ \phi_\ep$, where $\phi_\ep$ satisfies (\ref{phi-ep'}). Recall from (\ref{L-ep}) that $L_\ep g = -v_\ep g_x + Hg + (u_\ep g)_x$, so
\[
h_t = -v_\ep h_x + Hh + u_\ep h_x - \ep'(t)\partial_\ep U_\ep + (v_\ep - a'(t))(u_\ep + h_x) + \frac{1}{2}h_x^2 \pmod 1
\]
where $U_\ep$ is a primitive of $u_\ep$. Differentiating $h = \tilde h \circ \phi_\ep$ with respect to $\ep$ we get
\[
h_t = \tilde h_t \circ \phi_\ep + \ep'(t)(\partial_\ep\phi_\ep)(\tilde h_x \circ \phi_\ep).
\]
On the other hand,
\[
(-v_\ep h_x + Hh + u_\ep h_x)_x = L_\ep g
= (((c_\ep\partial_x + H + R_\ep)\tilde h) \circ \phi_\ep)_x
\]
so
\begin{align*}
\tilde h_t &= (c_\ep\partial_x + H + R_\ep)\tilde h - \ep'(t)(\partial_\ep\phi_\ep \circ \phi_\ep^{-1})\tilde h_x - \ep'(t)\partial_\ep U_\ep \circ \phi_\ep^{-1}\\
&+ (v_\ep - a'(t))(u_\ep + h_x) \circ \phi_\ep^{-1} + \frac{1}{2} h_x^2 \circ \phi_\ep^{-1} \pmod 1.
\end{align*}
By the chain rule, $h_x = \phi_\ep'(\tilde h_x \circ \phi_\ep)$, so
$h_x \circ \phi_\ep^{-1} = (\phi_\ep' \circ \phi_\ep^{-1})\tilde h_x$, and
\begin{align*}
\tilde h_t &= (c_\ep\partial_x + H + R_\ep)\tilde h + \Phi_\ep\tilde h_x - \ep'(t)\partial_\ep U_\ep \circ \phi_\ep^{-1} + (v_\ep - a'(t))u_\ep \circ \phi_\ep^{-1}\\
&+ \frac{1}{2} (\phi_\ep' \circ \phi_\ep^{-1})^2\tilde h_x^2 \pmod 1,\\
\Phi_\ep &= -\ep'(t)(\partial_\ep\phi_\ep \circ \phi_\ep^{-1})
+ (v_\ep - a'(t))(\phi_\ep' \circ \phi_\ep^{-1}).
\end{align*}

Using the operator $W_\ep$ from lemma \ref{conjugado} we have that
\[
(1 + W_\ep)(c_\ep\partial_x + H + R_\ep) = \Lambda_\ep(1 + W_\ep)
\]
where $\Lambda_\ep$ is a Fourier multiplier whose action on the Fourier mode $e^{i(n + \sgn n)x}$ is multiplication by $\lambda_n(\ep)$.
Since $W_\ep/\ep$ is of class $\mathcal S$, uniformly in $\ep$,
for any smooth function $F$, the operator
\[
\tilde h \mapsto \mathcal R_\ep(F)\tilde h
:= (1 + W_\ep)(F\tilde h_x) - F((1 + W_\ep)\tilde h)_x
\]
is of class $\mathcal S$, with the implicit constants depending on the $C^k$ norms of $F$.

Let $\frak h = (1 + W_\ep)\tilde h$. Then
\begin{align*}
(1 + W_\ep)\tilde h_t &= \Lambda_\ep\frak h + \Phi_\ep\frak h_x
- \ep'(t)(1 + W_\ep)(\partial_\ep U_\ep \circ \phi_\ep^{-1})\\
&+ (v_\ep - a'(t))(1 + W_\ep)(u_\ep \circ \phi_\ep^{-1}) + N_\ep[\frak h, \frak h] + \mathcal R_\ep(\Phi_\ep)\tilde h \pmod 1
\end{align*}
where
\begin{equation}\label{N-def}
N_\ep[\frak h, \frak h] = \frac{1}{2} (\phi_\ep' \circ \phi_\ep^{-1})^2((1 + W_\ep)^{-1}\frak h)_x^2.
\end{equation}
Both $\mathcal R_\ep(\partial_\ep\phi_\ep \circ \phi_\ep^{-1})$ and $\mathcal R_\ep(\phi_\ep' \circ \phi_\ep^{-1} - 1)/\ep$ are of class $\mathcal S$, uniformly in $\ep$ when $\ep$ is small. Moreover, since $W_\ep$ is analytic in $\ep$ with $W_0 = 0$, so is $\mathcal R_\ep(1)$ with $\mathcal R_0(1) = 0$. Hence $R_\ep(1)/\ep$ is of class $\mathcal S$ uniformly in $\ep$, and so is $\mathcal R_\ep(\phi_\ep' \circ \phi_\ep^{-1})/\ep$.

Since $\partial_\ep u_\ep$ and $u_\ep'$ are in the generalized eigenspace of $L_\ep$ associated with the eigenvalue 0, we have that $\partial_\ep U_\ep \circ \phi_\ep^{-1}$ and $u_\ep \circ \phi_\ep^{-1}$ are in the corresponding space of $c_\ep\partial_x + H + R_\ep$, so $(1 + W_\ep)(\partial_\ep U_\ep \circ \phi_\ep^{-1})$ and $(1 + W_\ep)(u_\ep \circ \phi_\ep^{-1})$ are in the space spanned by $\sin x$ and $\cos x$, according to Lemma \ref{conjugado}.

Now we have
\begin{equation}\label{frak-h-t}
\begin{aligned}
\frak h_t &= (1 + W_\ep)\tilde h_t + \ep'(t)\partial_\ep W_\ep \tilde h\\
&= \Lambda_\ep\frak h + \Phi_\ep\frak h_x + N_\ep[\frak h, \frak h] + \text{Rest}
\pmod{1, \sin x, \cos x}
\end{aligned}
\end{equation}
where $N_\ep[\frak h, \frak h]$ is given by (\ref{N-def}) and
\[
\text{Rest} = \ep'(t)(\partial_\ep W_\ep)\tilde h + \mathcal R_\ep(\Phi_\ep)\tilde h
\]
is also of class $\mathcal S$ uniformly in $\ep$ when $\ep$ is small.

Recall that $\ep'(t)$ and $a'(t)$ are chosen such that $P_0(\ep) g(t) = 0$ for all $t$, where $P_0(\ep)$ is the projection onto the span of $\partial_\ep u_\ep$ and $u_\ep'$. This implies that $Q_0(\ep) \tilde h(t) = 0$ for all $t$, where $Q_0(\ep)$ is the projection onto the span of $\partial_\ep U_\ep \circ \phi_\ep^{-1}$ and $u_\ep \circ \phi_\ep^{-1}$. Since $1 + W_\ep$ maps the span of $\partial_\ep U_\ep \circ \phi_\ep^{-1}$ and $u_\ep \circ \phi_\ep^{-1}$ to the span of $\sin x$ and $\cos x$, we have that $\hat{\frak h}(1) = \hat{\frak h}(-1) = 0$ for all $t$.

\section{Energy estimates}\label{EE}
Since $\hat{\frak h}(1) = \hat{\frak h}(-1) = 0$ for all $t$,
for $k = 0, 1, \dots$  we define the energy
\[
E_k = \frac{1}{2}\|\frak h\|^2_{\dot H^k}
= \frac{1}{2}\|\frak h\|^2_{H^k/(1, \sin x, \cos x)}
\]
and aim to control its growth.

Using the evolution equation (\ref{frak-h-t}) for $\frak h$ and the anti-selfadjointness of $\Lambda_\ep$ we get
\begin{align*}
\frac{d}{dt}E_k(t) &= E_\Phi(t) + E_N(t) + E_{\text{Rest}}(t),\\
E_\Phi(t) &= \langle \Phi_\ep\frak h_x, \frak h \rangle_{\dot H^k}\\
E_N(t) &= \langle N_\ep[\frak h(t), \frak h(t)], \frak h(t) \rangle_{\dot H^k},\\
E_{\text{Rest}}(t) &= \langle \ep'(t)\partial_\ep W_\ep\tilde h(t)
+ \mathcal R_\ep(\Phi_\ep)\tilde h(t), \frak h(t) \rangle_{\dot H^k}.
\end{align*}

Recall that $g = h_x$, $h = \tilde h \circ \phi_\ep$ and $\frak h = (1 + W_\ep)\tilde h$. When $\ep$ is small enough, the last two are bounded operators with bounded inverse between $\dot H^k$, $k = 0, 1, \dots$, so
\begin{equation}\label{gHk-hHk}
\|g\|_{H^k} \approx_k \|h\|_{\dot H^{k+1}} \approx_k \|\tilde h\|_{\dot H^{k+1}} \approx_k \|\frak h\|_{\dot H^{k+1}}.
\end{equation}
Since $\mathcal R_\ep(\partial_\ep\phi_\ep \circ \phi_\ep^{-1})$, $\mathcal R_\ep(\phi_\ep' \circ \phi_\ep^{-1})/\ep$ and $\partial_\ep W_\ep$ are of class $\mathcal S$ uniformly in $\ep$,
\begin{align*}
\|(\ep'(t)\partial_\ep W_\ep - \ep'(t)\mathcal R_\ep(\partial_\ep\phi_\ep \circ \phi_\ep^{-1}))\tilde h(t)\|_{\dot H^k}
\lesssim_k \|g(t)\|_{H^1}^2\|\tilde h(t)\|_{\dot H^1}
&\lesssim_k E_2(t)^{3/2},\\
\|(v_\ep - a'(t))\mathcal R_\ep(\phi_\ep' \circ \phi_\ep^{-1})\tilde h(t)\|_{\dot H^k}
\lesssim_k (\|g(t)\|_{H^1}^2/\ep)\ep\|\tilde h(t)\|_{\dot H^1}
&\lesssim_k E_2(t)^{3/2}
\end{align*}
so
\begin{equation}\label{ER-bound}
|E_{\text{Rest}}(t)| \lesssim_k E_2(t)^{3/2}E_k(t)^{1/2}.
\end{equation}

To bound $E_\Phi$ we use (\ref{ep'-a'-bound}) and (\ref{gHk-hHk}) to get
\[
\|\Phi_\ep'\|_{C^k} \lesssim_k \|g(t)\|_{H^1}^2 + (\|g(t)\|_{H^1}^2/|\ep|)|\ep| \lesssim_k E_2(t).
\]
Since $E_\Phi$ loses only one derivative in $\frak h$, we have that
\begin{equation}\label{Phi-Phi0}
|E_\Phi(t) - \langle \Phi_\ep\partial_x^{k+1}\frak h(t), \partial_x^k\frak h(t) \rangle_{L^2/(1)}| \lesssim_k E_2(t)E_k(t).
\end{equation}
For the sake of bounding this term, since the inner product is taken in the space $L^2/(1)$, we can without loss of generality assume that $\hat{\frak h}(0) = 0$ (which is not true in general) and integrate by parts to get
\[
2\langle \Phi_\ep\partial_x^{k+1}\frak h(t), \partial_x^k\frak h(t) \rangle_{L^2/(1)} = \int_0^{2\pi} \Phi_\ep\partial_x(\partial_x^k \frak h(t))^2dx = -\int_0^{2\pi} \Phi_\ep'(\partial_x^k \frak h(t))^2dx
\]
so again by (\ref{ep'-a'-bound}) and (\ref{gHk-hHk}),
\begin{equation}\label{Phi0-bound}
|E_\Phi(t)| \lesssim_k E_2(t)E_k(t).
\end{equation}


Combining (\ref{ER-bound}), (\ref{Phi-Phi0}) and (\ref{Phi0-bound}) shows that
\begin{equation}\label{Ek-EN-bound}
\left| \frac{d}{dt}E_k(t) - E_N(t) \right| \lesssim_k E_2(t)E_k(t).
\end{equation}

\subsection{Normal form transformation}
To bound $E_N$ we recall the expression of $N_\ep$ from (\ref{N-def}).
Since $N_\ep$ does not depend on the constant mode of $\frak h$,
we can also assume without loss of generality that $\hat{\frak h}(0) = 0$.
We further decompose
\begin{equation}\label{EN12}
\begin{aligned}
E_N(t) &= E_{N1}(t) + E_{N2}(t),\\
E_{N1}(t) &= \frac{1}{2}\int_0^{2\pi} \partial_x^k\frak h(t)
\partial_x^k(\partial_x\frak h(t))^2dx\\
&= \sum_{j=2}^{[k/2]+1} c_{kj}\int_0^{2\pi} \partial_x^k\frak h(t)
\partial_x^{k+2-j}\frak h(t)\partial_x^j\frak h(t)
\end{aligned}
\end{equation}
where $c_{kj} \in \R$ are constants and we integrated by parts to get rid of the terms with $k + 1$ derivatives falling on a single factor of $\frak h$.

We use the normal form transformation to bound them. Define the trilinear map
\begin{align*}
D_\ep[f_1, f_2, f_3]
&= \sum_{mnl\neq0} \frac{1}{\lambda_m(\ep) + \lambda_n(\ep) + \lambda_l(\ep)}
\int_0^{2\pi} \hat f_1(m + \sgn m)e^{i(m + \sgn m)x}\\
&\times \hat f_2(n + \sgn n)e^{i(n + \sgn n)x}
\hat f_3(l + \sgn l)e^{i(l + \sgn l)x}dx
\end{align*}
and put
\[
D_{1,k,j}(t) = D_{\ep(t)}[\partial_x^k\frak h(t), \partial_x^{k+2-j}\frak h(t), \partial_x^j\frak h(t)].
\]
Then
\begin{align*}
\frac{d}{dt}D_{1,k,j}(t)
&= \ep'(t)(\partial_\ep D_\ep)[\partial_x^k\frak h(t), \partial_x^{k+2-j}\frak h(t), \partial_x^j\frak h(t)]\\
&+ D_\ep[\partial_x^k\partial_t\frak h(t), \partial_x^{k+2-j}\frak h(t), \partial_x^j\frak h(t)]\\
&+ D_\ep[\partial_x^k\frak h(t), \partial_x^{k+2-j}\partial_t\frak h(t), \partial_x^j\frak h(t)]\\
&+ D_\ep[\partial_x^k\frak h(t), \partial_x^{k+2-j}\frak h(t), \partial_x^j\partial_t\frak h(t)].
\end{align*}
Note that $E_{N1}(t)$ is a linear combination of the last three lines on the right-hand side, with $\partial_t$ replaced with $\Lambda_\ep$,
so $\frac{d}{dt}\sum_{j=2}^{[k/2]+1} c_{jk}D_{1,k,j}(t) - E_{N1}(t)$
is a linear combination of
\begin{align}
\label{IBP0}
&\ep'(t)(\partial_\ep D_\ep)[\partial_x^k\frak h(t), \partial_x^{k+2-j}\frak h(t), \partial_x^j\frak h(t)],\\
\label{IBP1}
&D_\ep[\partial_x^k(\partial_t - \Lambda_\ep)\frak h(t), \partial_x^{k+2-j}\frak h(t), \partial_x^j\frak h(t)],\\
\label{IBP2}
&D_\ep[\partial_x^k\frak h(t), \partial_x^{k+2-j}(\partial_t - \Lambda_\ep)\frak h(t), \partial_x^j\frak h(t)],\\
\label{IBP3}
&D_\ep[\partial_x^k\frak h(t), \partial_x^{k+2-j}\frak h(t), \partial_x^j(\partial_t - \Lambda_\ep)\frak h(t)].
\end{align}
We estimate these terms one by one.

By the definition of $D_\ep$,
\begin{align*}
(\ref{IBP0}) &= \ep'(t)\sum_{mnl\neq0} \frac{(\lambda_m'(\ep) + \lambda_n'(\ep) + \lambda_l'(\ep))}{2(\lambda_m(\ep) + \lambda_n(\ep) + \lambda_l(\ep))^2}
\int_0^{2\pi} \hat{\frak h}(m + \sgn m, t)\partial_x^ke^{i(m + \sgn m)x}\\
&\times \hat{\frak h}(n + \sgn n, t)\partial_x^{k+2-j}e^{i(n + \sgn n)x}
\hat{\frak h}(l + \sgn l, t)\partial_x^je^{i(l + \sgn l)x}dx.
\end{align*}
We first bound the fraction. By Corollary \ref{lambda-remainder},
when $\ep$ is small enough,
\begin{align}
\nonumber
\lambda_m'(\ep) + \lambda_n'(\ep) + \lambda_l'(\ep)
&= (m + n + l + \sgn m + \sgn n + \sgn l)\partial_\ep c_\ep i + O(\ep^5)\\
&\lesssim (|m + n + l| + 1)|\ep|.
\label{sum-lambda'}
\end{align}
On the other hand, the integral vanishes unless
\begin{equation}\label{3freq}
m + n + l + \sgn m + \sgn n + \sgn l = 0
\end{equation}
in which case $m + n + l$ is an odd number, so is non-zero. Then by Case 1 of Proposition \ref{non-res},
\begin{equation}\label{sum-lambda-non-res}
|\lambda_m(\ep) + \lambda_n(\ep) + \lambda_l(\ep)| > |m + n + l|/2
\end{equation}
so
\begin{equation}\label{coef-non-res}
\frac{|\lambda_m'(\ep) + \lambda_n'(\ep) + \lambda_l'(\ep)|}
{|\lambda_m(\ep) + \lambda_n(\ep) + \lambda_l(\ep)|^2}
\lesssim |\ep|.
\end{equation}
Then for $k \ge 3$,
\begin{equation}\label{trilinear-bound}
\begin{aligned}
|(\ref{IBP0})|
&\lesssim |\ep(t)\ep'(t)|\sum_{mnl\neq0\atop(\ref{3freq})} |(m + \sgn m)^k\hat{\frak h}(m + \sgn m, t)\\
&\times (n + \sgn n)^{k+2-j}\hat{\frak h}(n + \sgn n, t)(l + \sgn l)^j\hat{\frak h}(l + \sgn l, t)|\\
&\approx |\ep(t)\ep'(t)|\left| \int_0^{2\pi} \partial_x^kH(x, t)\partial_x^{k+2-j}H(x, t)\partial_x^jH(x, t) \right|dx\\
&\lesssim_k |\ep(t)\ep'(t)|\|H(x, t)\|_{H_x^k}^2\|H(x, t)\|_{W_x^{[k/2]+1,\infty}}
\lesssim_k |\ep(t)\ep'(t)|\|H(x, t)\|_{H_x^k}^3
\end{aligned}
\end{equation}
since $k \ge [k/2] + 2$, where
\[
H(x, t) = \sum_{m\neq0} |\hat{\frak h}(m + \sgn m, t)|e^{i(m + \sgn m)x}
\]
satisfies
\[
\|H(x, t)\|_{H_x^k} = \|\frak h(t)\|_{\dot H^k} \lesssim E_k(t)^{1/2}
\]
so by (\ref{ep'-a'-bound}) and (\ref{gHk-hHk}),
\begin{equation}\label{IBP0-bound}
|(\ref{IBP0})| \lesssim_k |\ep|E_2(t)E_k(t)^{3/2}.
\end{equation}

To bound the other terms \eqref{IBP1}, \eqref{IBP2} and \eqref{IBP3}, we use the evolution equation (\ref{frak-h-t}) of $\frak h$, which loses one derivative in $\frak h$, so
\[
\|(\partial_t - \Lambda_\ep)\frak h(t)\|_{\dot H^{k-1}}
\lesssim (\|g\|_{H^1}^2/|\ep|)\|\frak h(t)\|_{\dot H^k} + \|\frak h(t)\|_{\dot H^k}^2.
\]
If $\|g(t)\|_{H^1}/|\ep|$ is small enough and $k \ge 2$,
the first term is dominated by the second term thanks to (\ref{gHk-hHk}).
Since in the summation of $D_{\ep}$ it holds that $m + n + l \neq 0$,
the denominator is uniformly bounded from below thanks to (\ref{sum-lambda-non-res}). Unless $j = 2$ in (\ref{IBP1}) and (\ref{IBP2}),
we can integrate by parts if necessary to ensure that at most $k - 1$ derivatives in $x$ hit each factor of $j$. Then similarly to (\ref{trilinear-bound}) it follows that for $k \ge 5$,
\begin{equation}\label{IBP3-bound}
|(\ref{IBP1}, j \ge 3) + (\ref{IBP2}, j \ge 3) + (\ref{IBP3})| \lesssim_k E_k(t)^2.
\end{equation}
For $j = 2$, by symmetry of $D_{\ep}$ it is clear that
\begin{equation}\label{IBP2-bound}
(\ref{IBP2}, j = 2) = (\ref{IBP1}, j = 2)
\end{equation}
which according to (\ref{frak-h-t}) equals
\[
D_\ep[\partial_x^k(\Phi_\ep\frak h_x(t) + N_\ep[\frak h(t), \frak h(t)] + R(t)), \partial_x^k\frak h(t), \partial_x^2\frak h(t)].
\]
Similarly to (\ref{trilinear-bound}),
\begin{equation}\label{IBP1R-bound}
|D_\ep[\partial_x^kR(t), \partial_x^k\frak h(t), \partial_x^2\frak h(t)]|
\lesssim_k E_3(t)^2E_k(t)^{1/2}.
\end{equation}
Similarly to (\ref{Phi-Phi0}),
\begin{equation}\label{IBP1Phi0}
|D_\ep[\partial_x^k(\Phi_\ep\frak h_x(t)) - \Phi_\ep\partial_x^{k+1}\frak h(t), \partial_x^k\frak h(t), \partial_x^2\frak h(t)]|
\lesssim_k E_3(t)^{3/2}E_k(t).
\end{equation}
By definition of $D_{\ep}$,
\begin{equation}\label{IBP1Phi}
\begin{aligned}
&D_\ep[\Phi_\ep\partial_x^{k+1}\frak h_x(t), \partial_x^k\frak h(t), \partial_x^2\frak h(t)]
= \sum_{mm'nl\neq0} \frac{1}{\lambda_{m'}(\ep) + \lambda_n(\ep) + \lambda_l(\ep)}\\
\times &\int_0^{2\pi} \hat\Phi_\ep(p)e^{ipx}\hat{\frak h}(m + \sgn m, t)\partial_x^{k+1}e^{i(m + \sgn m)x}\\
\times &\hat{\frak h}(n + \sgn n, t)\partial_x^ke^{i(n + \sgn n)x}
\hat{\frak h}(l + \sgn l, t)\partial_x^2e^{i(l + \sgn l)x}dx
\end{aligned}
\end{equation}
where $m' + \sgn m' = p + m + \sgn m \neq 0, \pm1$.
We break the summation into several parts.

{\bf Part 1:} $|p| \ge |m + \sgn m|/3$.
Then we can transfer the extra derivative from $\frak h$ to $\Phi_\ep$,
and compute as in (\ref{Phi-Phi0}) to get
\begin{equation}\label{IBP1Phi1}
|\textbf{Part 1}| \lesssim_k E_3(t)^{3/2}E_k(t).
\end{equation}

{\bf Part 2:} $|p| < |m + \sgn m|/3$ but $|p| \ge |n + \sgn n|/3$.
If $|n + \sgn n| \ge |m|/3$ then $|p| \ge |m|/9$, and we get the same bound as before. Otherwise, since the integral vanishes unless
\begin{equation}\label{4freq}
p + m + n + l + \sgn m + \sgn n + \sgn l = 0
\end{equation}
in which case we have $|l + \sgn l| > |n + \sgn n|/3$, and we can transfer the extra derivative to the factor $\partial_x^2\frak h$ to get (note that $\|\Phi_\ep\|_{C^k} \lesssim_k \|g\|_{H^1}^2/|\ep|$)
\begin{equation}\label{IBP1Phi2}
|\textbf{Part 2}| \lesssim_k (\|g(t)\|_{H^1}^2/|\ep|)E_4(t)^{1/2}E_k(t)
\lesssim E_4(t)E_k(t)
\end{equation}
provided that $\|g(t)\|_{H^1}/|\ep|$ is small enough.

{\bf Part 3:} $|p| < |m + \sgn m|/3$ and $|p| < |n + \sgn n|/3$.
Then $\sgn(m' + \sgn m') = \sgn(m + \sgn m)$, i.e., $\sgn m' = \sgn m$,
so $m' = m + p$. By symmetry,
\begin{align*}
&\textbf{Part 3}\\
= &\sum_{mnl\neq0\atop{|p|<|m + \sgn m|/3 \atop |p|<|n + \sgn n|/3}}
\frac{1/2}{\lambda_{m+p}(\ep) + \lambda_n(\ep) + \lambda_l(\ep)}
\int_0^{2\pi} \hat{\frak h}(l + \sgn l, t)\partial_x^2e^{i(l + \sgn l)x}\\
\times &\hat\Phi_\ep(p)e^{ipx}\hat{\frak h}(m + \sgn m, t)\partial_x^{k+1}e^{i(m + \sgn m)x}\hat{\frak h}(n + \sgn n, t)\partial_x^ke^{i(n + \sgn n)x}dx\\
+ &\sum_{mnl\neq0\atop{|p|<|m + \sgn m|/3 \atop |p|<|n + \sgn n|/3}}
\frac{1/2}{\lambda_m(\ep) + \lambda_{n+p}(\ep) + \lambda_l(\ep)}
\int_0^{2\pi} \hat{\frak h}(l + \sgn l, t)\partial_x^2e^{i(l + \sgn l)x}\\
\times &\hat\Phi_\ep(p)e^{ipx}\hat{\frak h}(m + \sgn m, t)\partial_x^ke^{i(m + \sgn m)x}\hat{\frak h}(n + \sgn n, t)\partial_x^{k+1}e^{i(n + \sgn n)x}dx.
\end{align*}
Note that the two denominators are uniformly bounded from below. Also,
$\sgn(m + p) = \sgn m$ and $|m + p| > (2|m| - 1)/3$, and similarly for $l$.
Then by Corollary \ref{lambda-remainder}, the two denominators differ by $O(|m|\ep^{4|m|/3+3} + |n|\ep^{4|n|/3+3})$, so
\begin{align*}
&\textbf{Part 3}\\
= &\sum_{mnl\neq0\atop{|p|<|m + \sgn m|/3 \atop |p|<|n + \sgn n|/3}}
\frac{1/2}{\lambda_{m+p}(\ep) + \lambda_n(\ep) + \lambda_l(\ep)}
\int_0^{2\pi} \hat{\frak h}(l + \sgn l, t)\partial_x^2e^{i(l + \sgn l)x}\\
\times &\hat\Phi_\ep(p)e^{ipx}\partial_x(\hat{\frak h}(m + \sgn m, t)\partial_x^ke^{i(m + \sgn m)x}\hat{\frak h}(n + \sgn n, t)\partial_x^ke^{i(n + \sgn n)x})dx\\
+ &\sum_{(\ref{4freq})} O(|m|\ep^{4|m|/3+3} + |n|\ep^{4|n|/3+3})
\int_0^{2\pi} |(l + \sgn l)^2\hat{\frak h}(l + \sgn l, t)\\
\times &\hat\Phi_\ep(p)(m + \sgn m)^k\hat{\frak h}(m + \sgn m, t)(n + \sgn n)^{k+1}\hat{\frak h}(n + \sgn n, t)|\\
= &\sum_{mnl\neq0\atop{|p|<|m + \sgn m|/3 \atop |p|<|n + \sgn n|/3}}
\frac{-1/2}{\lambda_{m+p}(\ep) + \lambda_n(\ep) + \lambda_l(\ep)}
\int_0^{2\pi} \partial_x(\hat{\frak h}(l + \sgn l, t)\partial_x^2e^{i(l + \sgn l)x}\\
\times &\hat\Phi_\ep(p)e^{ipx})\hat{\frak h}(m + \sgn m, t)\partial_x^ke^{i(m + \sgn m)x}\hat{\frak h}(n + \sgn n, t)\partial_x^ke^{i(n + \sgn n)x}dx\\
+ &\sum_{(\ref{4freq})} O(\ep^4)
\int_0^{2\pi} |(l + \sgn l)^2\hat{\frak h}(l + \sgn l, t)\\
\times &\hat\Phi_\ep(p)(m + \sgn m)^k\hat{\frak h}(m + \sgn m, t)(n + \sgn n)^k\hat{\frak h}(n + \sgn n, t)|
\end{align*}
where we integrated by parts in the first integral and used the bounds $|m|\ep^{4|m|/3+3}$ and $|n(n + \sgn n)\ep^{4|n|/3+3}|\lesssim\ep^4$ in the second. Then as in (\ref{trilinear-bound}) it follows that
\begin{equation}\label{IBP1Phi3}
\begin{aligned}
|\textbf{Part 3}|
&\lesssim_k (\|g(t)\|_{H^1}^2/|\ep|)E_4(t)^{1/2}E_k(t)
+ \ep^4(\|g(t)\|_{H^1}^2/|\ep|)E_3(t)^{1/2}E_k(t)\\
&\lesssim E_4(t)E_k(t)
\end{aligned}
\end{equation}
provided that $\ep$ and $\|g(t)\|_{H^1}/|\ep|$ are small enough.

Combining (\ref{IBP1Phi0}), (\ref{IBP1Phi1}), (\ref{IBP1Phi2}) and (\ref{IBP1Phi3}) shows that
\begin{equation}\label{IBP1Phi-bound}
|D_\ep[\partial_x^k(\Phi_\ep\frak h_x(t)), \partial_x^k\frak h(t), \partial_x^2\frak h(t)]|
\lesssim_k E_4(t)(1 + E_4(t)^{1/2})E_k(t)
\end{equation}
provided that $\ep$ and $\|g(t)\|_{H^1}/|\ep|$ are small enough.

We now turn to $D_{\ep}[\partial_x^kN_\ep[\frak h(t), \frak h(t)], \partial_x^k\frak h(t), \partial_x^2\frak h(t)]$. Similarly to (\ref{Phi-Phi0}),
\begin{equation}\label{IBP1N0-bound}
\begin{aligned}
|&D_\ep[\partial_x^kN_\ep[\frak h(t), \frak h(t)], \partial_x^k\frak h(t), \partial_x^2\frak h(t)]\\
- &D_\ep[(\phi_\ep' \circ \phi_\ep^{-1})^2\partial_x((1 + W_\ep)^{-1}\frak h(t))(\partial_x^{k+1}(1 + W_\ep)^{-1}\frak h(t)), \partial_x^k\frak h(t), \partial_x^2\frak h(t)]|\\
\lesssim_k &|\ep|E_3(t)^{1/2}E_k(t)^{3/2}.
\end{aligned}
\end{equation}
Since $W_\ep/\ep$ is of class $\mathcal S$ uniformly in $\ep$,
so is $((1 + W_\ep)^{-1} - 1)/\ep$, so
\begin{equation}\label{IBP1N1-bound}
\begin{aligned}
&|D_\ep[(\phi_\ep' \circ \phi_\ep^{-1})^2(\partial_x(1 + W_\ep)^{-1}\frak h(t))(\partial_x^{k+1}((1 + W_\ep)^{-1} - 1)\frak h(t)), \partial_x^k\frak h(t), \partial_x^2\frak h(t)]|\\
&\lesssim_k |\ep|E_3(t)^{3/2}E_k(t)^{1/2}.
\end{aligned}
\end{equation}
Finally, $D_\ep[(\phi_\ep' \circ \phi_\ep^{-1})^2(\partial_x(1 + W_\ep)^{-1}\frak h(t))(\partial_x^{k+1}\frak h(t)), \partial_x^k\frak h(t), \partial_x^2\frak h(t)]$ is of the same form as the left-hand side of (\ref{IBP1Phi}), so we trace the same argument to get
\begin{align*}
|\textbf{Part 1}| &\lesssim_k E_3(t)E_k(t),\\
|\textbf{Part 2}| &\lesssim_k E_4(t)E_k(t),\\
|\textbf{Part 3}| &\lesssim_k E_4(t)E_k(t) + \ep^4E_3(t)E_k(t)
\lesssim E_4(t)E_k(t)
\end{align*}
provided that $\ep$ is small enough. Hence
\begin{equation}\label{IBP1N2-bound}
|D_\ep[(\phi_\ep' \circ \phi_\ep^{-1})^2(\partial_x(1 + W_\ep)^{-1}\frak h(t))(\partial_x^{k+1}\frak h(t)), \partial_x^k\frak h(t), \partial_x^2\frak h(t)]| \lesssim_k E_4(t)E_k(t)
\end{equation}

Combining (\ref{IBP1N0-bound}), (\ref{IBP1N1-bound}) and (\ref{IBP1N2-bound}) shows that, for $k \ge 4$,
\begin{equation}\label{IBP1N-bound}
|D_\ep[\partial_x^kN_\ep[\frak h(t), \frak h(t)], \partial_x^k\frak h(t), \partial_x^2\frak h(t)]| \lesssim_k E_4(t)^{1/2}E_k(t)^{3/2}
\end{equation}
provided that $\ep$ is small enough.

Combining (\ref{IBP1R-bound}), (\ref{IBP1Phi-bound}) and (\ref{IBP1N-bound}) shows that, for $k \ge 4$,
\begin{equation}\label{IBP1-bound}
|(\ref{IBP1}, j = 2)|\lesssim_k E_4(t)^{1/2}(1 + E_4(t)^{1/2})E_k(t)^{3/2}.
\end{equation}
provided that $\ep$ and $\|g(t)\|_{H^1}/|\ep|$ are small enough.

Finally, combining (\ref{IBP0-bound}), (\ref{IBP3-bound}), (\ref{IBP2-bound}) and (\ref{IBP1-bound}) shows that, for $k \ge 5$,
\begin{equation}\label{EN1-bound}
\left| \frac{d}{dt}\sum_{j=2}^{[k/2]+1} c_{jk}D_{1,k,j}(t) - E_{N1}(t)\right| \lesssim_k (1 + E_4(t)^{1/2})E_k(t)^2
\end{equation}
provided that $\ep$ and $\|g(t)\|_{H^1}/|\ep|$ are small enough.

\subsection{Lifespan when $\delta \ll \ep$}\label{lifespan1}
In this section we will obtain a preliminary bound for $E_{N2} = E_N - E_{N1}$ and show a lifespan of $\frac{1}{\ep \delta}$ when $||g_0||_{H^5(\T)}=\delta\ll\ep$, i.e., $\delta\leq c\ep$ for some $c>0$ independent of $\ep$.

Recall from (\ref{EN12}) that
\[
E_N(t) = \frac{1}{2}\int_0^{2\pi} \partial_x^k\frak h(t)\partial_x^k((\phi_\ep' \circ \phi_\ep^{-1})^2((1 + W_\ep)^{-1}\frak h(t))_x^2)dx.
\]
Similarly to (\ref{Phi-Phi0}), for $k \ge 3$,
\[
\left| E_N(t) - \int_0^{2\pi} (\phi_\ep' \circ \phi_\ep^{-1})^2\partial_x^k\frak h(t)\partial_x^k(((1 + W_\ep)^{-1}\frak h_x(t))^2)dx \right|
\lesssim_k |\ep|E_k(t)^{3/2}.
\]
Since $((1 + W_\ep)^{-1} - 1)/\ep$ is of class $\mathcal S$ uniformly in $\ep$,
\[
\begin{aligned}
\left| \int_0^{2\pi} (\phi_\ep' \circ \phi_\ep^{-1})^2\partial_x^k\frak h(t)\partial_x^k(((1 + W_\ep)^{-1}\frak h_x(t) - \frak h_x(t))^2)dx \right|&\lesssim_k \ep^2E_k(t)^{3/2},\\
2\left| \int_0^{2\pi} (\phi_\ep' \circ \phi_\ep^{-1})^2\partial_x^k\frak h(t)\partial_x^k(((1 + W_\ep)^{-1}\frak h_x(t) - \frak h_x(t))\frak h_x(t))dx \right.&\\
- \left. \int_0^{2\pi} (\phi_\ep' \circ \phi_\ep^{-1})^2((1 + W_\ep)^{-1}\frak h_x(t) - \frak h_x(t))\partial_x^k\frak h(t)\partial_x^{k+1}\frak h(t)dx \right| &\lesssim_k |\ep|E_k(t)^{3/2}.
\end{aligned}
\]
Finally, by integration by parts,
\[
\begin{aligned}
\left| 2\int_0^{2\pi} (\phi_\ep' \circ \phi_\ep^{-1})^2((1 + W_\ep)^{-1}\frak h_x(t) - \frak h_x(t))\partial_x^k\frak h(t)\partial_x^{k+1}\frak h(t)dx \right|&\\
= \left| \int_0^{2\pi} \partial_x((\phi_\ep' \circ \phi_\ep^{-1})^2((1 + W_\ep)^{-1}\frak h_x(t) - \frak h_x(t)))(\partial_x^k\frak h(t))^2dx \right| &\lesssim_k |\ep|E_k(t)^{3/2}.
\end{aligned}
\]
Combining the bounds above shows that, for $k \ge 3$,
\begin{equation}\label{EN2-bound}
|E_{N_2}(t)| = |E_N(t) - E_{N_1}(t)| \lesssim_k |\ep|E_k(t)^{3/2}
\end{equation}
provided that $\ep$ is small enough.

Now combining (\ref{Ek-EN-bound}), (\ref{EN1-bound}) and (\ref{EN2-bound}) shows that, for $k \ge 5$,
\begin{equation}\label{Ek-bound}
\frac{d}{dt}\left| \sum_{j=2}^{[k/2]+1} c_{jk}D_{1,k,j}(t) - E_k(t)\right| \lesssim_k (1 + E_4(t)^{1/2})E_k(t)^2 + |\ep|E_k(t)^{3/2}.
\end{equation}
provided that $\ep$ and $\|g(t)\|_{H^1}/|\ep|$ are small enough. Hence
\begin{align*}
E_k(t) - E_k(0) &= \sum_{j=2}^{[k/2]+1} c_{jk}(D_{1,k,j}(t) - D_{1,k,j}(0))\\
&+ O_k(\|(1 + E_4^{1/2})E_k^2 + |\ep|E_k^{3/2}\|_{L^1([0,t])}).
\end{align*}
Similarly to (\ref{trilinear-bound}),
\[
|D_{1,k,j}(t)| = |D_{\ep(t)}[\partial_x^k\frak h(t), \partial_x^{k+2-j}\frak h(t), \partial_x^j\frak h(t)]| \lesssim_k E_k(t)^{3/2}.
\]

Now we are able to show a lifespan longer than what follows from local wellposedness. Assume that the initial data is
\[
f(x, 0) = u_\ep(x) + g(x)
\]
where $|\ep| \le \ep_0$ is small enough, the energy $E_k(0)$ computed from $g$ is $E_k(0) = \delta^2$, and $|\delta/\ep|$ is also small enough. Let
\begin{align}
\label{ft=u+g}
T^* = \sup\{&T: \exists\text{ a solution }f(x, t) = u_{\ep(t)}(x + a(t)) + g(x + a(t), t),\\
&t \in [0, T]\text{ such that }|\ep|/2 \le |\ep(t)| \le 2|\ep|,\ E_k(t) \le 4\delta^2\}.
\label{E-bound}
\end{align}
Then the above conditions hold for all $t < T^*$. Moreover, the energy estimate implies that
\[
E_k(t) = \delta^2 + O_k(\delta^3 + t(\delta^4 + |\ep|\delta^3))
= \delta^2 + O_k(\delta^3(1 + t|\ep|)).
\]
Then there is $c_k > 0$ such that if $T^* \le c_k/|\ep|\delta$, then $E_k(t) \le 2\delta^2$. Also,
\[
|\|f(x, t)\|_{L^2} - \|u_\ep\|_{L^2}|
= |\|f(x, 0)\|_{L^2} - \|u_\ep\|_{L^2}| \le \|g\|_{L^2} \lesssim \delta.
\]
by conservation of the $L^2$ norm. Meanwhile $|\|f(x, t)\|_{L^2} - \|u_{\ep(t)}\|_{L^2}| \lesssim \delta$, so $|\|u_{\ep(t)}\|_{L^2} - \|u_\ep\|_{L^2}| \lesssim \delta$. When $|\ep|$ is small enough, $\|u_\ep\|_{L^2}$ is differentiable in $\ep$ with nonzero derivative at $\ep = 0$. Since $|\delta/\ep|$ is small enough, $|\ep(t) - \ep| \lesssim \delta$.

By local wellposedness, the solution can be extended to a time $t^* > T^*$, with
\[
\|f(x, t) - f(x, T^*)\|_{H^2} \lesssim (t^* - T^*)(\|f(x, t)\|_{H^3} + \|f(x, t)\|_{H^3}^2) \le (t^* - T^*)|\ep|
\]
for $t \in [T^*, t^*]$. Then $\|f(x, t) - u_{\ep(T^*)}(x + a(T^*))\|_{H^2} \lesssim (t^* - T^*)|\ep| + \delta$. Take $t^* = T^* + \delta/|\ep|$. Then $f(x, t)$ satisfies the conditions in Proposition \ref{new-frame}, so (\ref{ft=u+g}) holds up to time $t^*$. Since $f(x, T^*)$ is small in $H^4$, $f(x, t)$ is uniformly bounded in $H^4$ on $[T^*, t^*]$, so it stays within a compact set in $H^2$. Since $\ep$ is differentiable in $f \in H^2$, $|\ep(t) - \ep(T^*)| \lesssim (t^* - T^*)|\ep| \lesssim \delta$, so $|\ep(t) - \ep| \lesssim \delta$, so $|\ep|/2 \le |\ep(t)| \le 2|\ep|$ holds up to time $t^*$. The energy estimate then implies that $E_k \le 3\delta^2$ also up to time $t^*$, so (\ref{E-bound}) holds up to time $t^*$, contradicting the definition of $T^*$. Hence the lifespan $T^* \gtrsim_k 1/|\ep|\delta$.

\subsection{Longer lifespan when $\delta \ll\ep^2$}
When the perturbation $g$ is very small  compare to $\ep^2$, i.e., $||g_0||_{H^5(\T)}=\delta\ll\ep^2$, we can obtain a longer lifespan by applying the normal form transformation to
\begin{align*}
E_{N2} &= E_N - E_{N1} = E_{N21} + E_{N22} + E_{N23} + E_{N24},\\
E_{N21}&= \sum_{j=1}^{[k/2]+1} c_{kj}'\int_0^{2\pi} \partial_x^k\frak h(t)\partial_x^{k+2-j}((\phi_\ep' \circ \phi_\ep^{-1})((1 + W_\ep)^{-1} - 1)\frak h(t))\\
&\times \partial_x^j((\phi_\ep' \circ \phi_\ep^{-1})(1 + W_\ep)^{-1}\frak h(t))dx,\\
E_{N22}&= \sum_{j=1}^{[k/2]+1} \sum_{i=1}^{k+2-j} c_{kji}\int_0^{2\pi} \partial_x^k\frak h(t)\partial_x^i(\phi_\ep' \circ \phi_\ep^{-1})\partial_x^{k+2-i-j}\frak h(t)\\
&\times \partial_x^j((\phi_\ep' \circ \phi_\ep^{-1})(1 + W_\ep)^{-1}\frak h(t))dx,\\
E_{N23}&= \sum_{j=2}^{[k/2]+1} c_{kj}\int_0^{2\pi} \partial_x^k\frak h(t)(\phi_\ep' \circ \phi_\ep^{-1} - 1)\partial_x^{k+2-j}\frak h(t)\\
&\times \partial_x^j((\phi_\ep' \circ \phi_\ep^{-1})(1 + W_\ep)^{-1}\frak h(t))dx,\\
E_{N24}&= \sum_{j=2}^{[k/2]+1} c_{kj}\int_0^{2\pi} \partial_x^k\frak h(t)\partial_x^{k+2-j}\frak h(t)\partial_x^j((\phi_\ep' \circ \phi_\ep^{-1})(1 + W_\ep)^{-1}\frak h - \frak h)dx
\end{align*}
where $c_{kj}$, $c_{kj}'$ and $c_{kji} \in \R$ are constants and we integrated by parts to get rid of the terms with $k + 1$ derivatives falling on a single factor of $\frak h$, except for the term with $j = 1$ in $E_{N21}$, in which the $k + 1$ derivatives do not matter in view of the fact that the operator $(\phi_\ep' \circ \phi_\ep^{-1})((1 + W_\ep)^{-1} - 1)$ is of class $\mathcal S$.

Now we define
\[
\begin{aligned}
D_{\ep,21}[f_1, f_2, f_3] &= \sum_{j=1\atop mnl\neq0}^{[k/2]+1} \frac{c_{kj}'}{\lambda_m(\ep) + \lambda_n(\ep) + \lambda_l(\ep)}\int_0^{2\pi} \hat f_1(m + \sgn m)e^{i(m + \sgn m)x}\\
&\times \partial_x^{k+2-j}((\phi_\ep' \circ \phi_\ep^{-1})((1 + W_\ep)^{-1} - 1)\hat f_2(n + \sgn n)e^{i(n + \sgn n)x})\\
&\times \partial_x^j((\phi_\ep' \circ \phi_\ep^{-1})(1 + W_\ep)^{-1}\hat f_3(l + \sgn l)e^{i(l + \sgn l)x}),\\
D_{21}(t) &= D_{\ep,21}[\frak h(t), \frak h(t), \frak h(t)]
\end{aligned}
\]
and similarly define $D_{22}$, $D_{23}$ and $D_{24}$. Then
\begin{align}
\label{IBP210}
\frac{d}{dt}D_{21}(t) - E_{N21}(t)
&= \ep'(t)(\partial_\ep D_{\ep,21})[\frak h(t), \frak h(t), \frak h(t)]\\
\label{IBP211}
&+ D_{\ep,21}[(\partial_t - \Lambda_\ep)\frak h(t), \frak h(t), \frak h(t)]\\
\label{IBP212}
&+ D_{\ep,21}[\frak h(t), (\partial_t - \Lambda_\ep)\frak h(t), \frak h(t)]\\
\label{IBP213}
&+ D_{\ep,21}[\frak h(t), \frak h(t), (\partial_t - \Lambda_\ep)\frak h(t)].
\end{align}
We estimate these terms one by one.

For (\ref{IBP210}), (\ref{sum-lambda'}) still holds, but there are non-trivial actions on $\frak h$ in the slots, so no frequency restriction such as (\ref{3freq}) exists. When $m + n + l \neq 0$, we are in Case 1 of Proposition \ref{non-res}, so (\ref{sum-lambda-non-res}), and hence (\ref{coef-non-res}), still hold. When $m + n + l = 0$, by Case 2 of Proposition \ref{non-res}, when $\ep$ is small enough,
\begin{equation}\label{sum-lambda-res}
|\lambda_m(\ep) + \lambda_n(\ep) + \lambda_l(\ep)| > \ep^2/5
\end{equation}
which, combined with (\ref{sum-lambda'}), shows that the multiplier in $\partial_\ep D_\ep$ is bounded by
\begin{equation}\label{coef-res}
\frac{|\lambda_m'(\ep) + \lambda_n'(\ep) + \lambda_l'(\ep)|}
{|\lambda_m(\ep) + \lambda_n(\ep) + \lambda_l(\ep)|^2}
\lesssim |\ep|^{-3}
\end{equation}
instead of (\ref{coef-non-res}). Since both $(\phi_\ep' \circ \phi_\ep^{-1})((1 + W_\ep)^{-1} - 1)/\ep$ and $\partial_\ep(\phi_\ep' \circ \phi_\ep^{-1})((1 + W_\ep)^{-1} - 1)$ are of class $\mathcal S$ uniformly in $\ep$, it follows that, for $k \ge 3$,
\begin{equation}\label{IBP210-bound}
|(\ref{IBP210})| \lesssim_k |\ep'(t)|\ep^{-2}E_k(t)^{3/2}
\lesssim \ep^{-2}E_2(t)E_k(t)^{3/2}
\end{equation}
provided that $\ep$ is small enough.

The terms (\ref{IBP211}), (\ref{IBP212}) and (\ref{IBP213}) are like (\ref{IBP1}), (\ref{IBP2}) and (\ref{IBP3}) respectively,
execpt that instead of the uniform lower bound of $\lambda_m(\ep) + \lambda_n(\ep) + \lambda_l(\ep)$ we now have (\ref{sum-lambda-res}),
which loses two factors of $\ep$, but we are helped by the $\ep$-smallness of $(\phi_\ep' \circ \phi_\ep^{-1})((1 + W_\ep)^{-1} - 1)$, which wins back a factor of $\ep$. All told we lose a factor of $\ep$ compared to (\ref{EN1-bound}), so for $k \ge 5$,
\begin{equation}\label{IBP211-bound}
|(\ref{IBP211}) + (\ref{IBP212}) + (\ref{IBP213})|
\lesssim_k |\ep|^{-1}(1 + E_4(t)^{1/2})E_k(t)^2
\end{equation}
provided that $\ep$ and $\|g(t)\|_{H^1}/|\ep|$ are small enough.

Combining (\ref{IBP210-bound}) and (\ref{IBP211-bound}) shows that for $k \ge 5$,
\begin{equation}\label{IBP21-bound}
\left| \frac{d}{dt}D_{21}(t) - E_{N21}(t) \right|
\lesssim_k |\ep|^{-1}(1 + E_4(t)^{1/2})E_k(t)^2
\end{equation}
provided that $\ep$ and $\|g(t)\|_{H^1}/|\ep|$ are small enough.
We can also save a factor of $\ep$ in the other terms $E_{N22}$, $E_{N23}$ and $E_{N24}$ thanks to the $\ep$-smallness of $(\phi_\ep' \circ \phi_\ep^{-1})'$ and $\phi_\ep' \circ \phi_\ep^{-1} - 1$. Hence the bound (\ref{IBP21-bound} also holds for $E_{N22}$, $E_{N23}$ and $E_{N24}$.

Combining (\ref{Ek-EN-bound}), (\ref{EN1-bound}) and (\ref{IBP21-bound}) shows that, for $k \ge 5$,
\begin{align*}
E_k(t) - E_k(0) &= \sum_{j=2}^{[k/2]+1} c_{jk}(D_{1,k,j}(t) - D_{1,k,j}(0))
+ \sum_{j=1}^4 (D_{2j}(t) - D_{2j}(0))\\
&+ O_k(|\ep|^{-1}\|(1 + E_4^{1/2})E_k^2\|_{L^1([0,t])})
\end{align*}
provided that $\ep$ and $\|g(t)\|_{H^1}/|\ep|$ are small enough.
Similarly to (\ref{EN2-bound}), for $k \ge 3$,
\[
|D_{2,k,j}(t)| \lesssim_k \ep(t)^{-2}|\ep(t)|E_k(t)^{3/2} = E_k(t)^{3/2}/|\ep|.
\]
Hence if $E_k(0) = \delta^2 \lesssim 1$ and $E_k \le 2\delta^2$ on $[0, t]$ then
\[
E_k(t) = \delta^2 + |\ep|^{-1}\delta^3 + O_k(t|\ep|^{-1}\delta^4).
\]
Assume $\delta/\ep^2$ is small. Then the second term on the right-hand side $\lesssim \delta^{5/2}$, so we close the estimate for a time $t \lesssim_k |\ep|/\delta^2$, which is also the lifespan in this case.

\section*{Acknowledgements}
 This work is supported in part by the Spanish Ministry of Science and Innovation,
through the “Severo Ochoa Programme for Centres of Excellence in R$\&$D”
(CEX2019-000904-S)” and MTM2017-89976-P. AC was partially supported by the Europa Excelencia program ERC2018-092824. DC and FZ were partially supported by the ERC Advanced Grant 788250.

\bibliographystyle{abbrv}


\end{document}